\documentclass{article}
\usepackage{array}
\usepackage{calrsfs}

\usepackage{amssymb,latexsym,amsthm}
\usepackage{amsrefs}
\usepackage{amsmath,amsfonts}
\usepackage{enumerate}

\usepackage{geometry}
\usepackage{url}
\ifdefined\directlua
\usepackage{fontspec}

\usepackage{polyglossia}
\setmainlanguage{english}
\usepackage{microtype}
\usepackage{pdftexcmds}
\makeatletter
\ifcase\pdf@shellescape
\relax\or
\begingroup
  \catcode`\%=12\relax
  \gdef\sformat{"Date: 
\endgroup
\directlua{
 local cmd="git show -s --format='"..\sformat.."'"
 local r=io.popen(cmd):read("*a")
 if (r) then
      tex.print("\string\\def\string\\COMMIT{"..r.."}")
 end
 }
\or
\relax\fi
\makeatother
\ifdefined\COMMIT
        \usepackage{background}
        \backgroundsetup{%
         pages=all, placement=bottom,angle=0,scale=2,%
         vshift=20pt,%
         contents={\COMMIT}}
\fi
\else
\usepackage[british]{babel}
\fi

\newcommand{\KK}{\mathbb{K}}
\newcommand{\RR}{\mathbb{R}}
\newcommand{\SSs}{\mathbb{S}}
\newcommand{\CC}{\mathbb{C}}
\newcommand{\cS}{\mathcal{S}}
\newcommand{\cQ}{\mathcal{Q}}
\newcommand{\cB}{\mathcal B}
\newcommand{\cA}{\mathcal A}
\newcommand{\cU}{\mathcal U}

\newcommand{\wt}{\mathrm{wt}}
\newcommand{\cL}{\mathcal{L}}
\newcommand{\cH}{\mathcal{H}}
\newcommand{\GG}{\mathbb{G}}
\newcommand{\HH}{\mathbb{H}}
\newcommand{\PP}{\mathbb{P}}
\newcommand{\XX}{\mathbb{X}}
\newcommand{\fB}{\mathfrak{B}}
\newcommand{\fX}{\mathfrak{X}}
\newcommand{\fP}{\mathfrak{P}}
\newcommand{\fL}{\mathfrak{L}}
\newcommand{\fA}{\mathfrak{A}}
\newcommand{\fC}{\mathfrak{C}}
\newcommand{\Gr}{\mathrm{Gr}}
\newcommand{\cM}{\mathcal{M}}
\newcommand{\cC}{\mathcal{C}}
\newcommand{\cV}{\mathcal{V}}
\newcommand{\cW}{\mathcal{W}}
\newcommand{\cT}{\mathcal{T}}
\newcommand{\cX}{\mathcal{X}}
\newcommand{\cF}{\mathcal{F}}
\newcommand{\cE}{\mathcal{E}}
\newcommand{\cY}{\mathcal{Y}}
\newcommand{\cP}{\mathcal{P}}
\newcommand{\ve}{\varepsilon}
 \newcommand{\cN}{\mathcal{N}}
\newcommand{\fa}{\mathfrak a}
\newcommand{\sh}{\mathrm{sh}}
\newcommand{\fb}{\mathfrak b}
\newcommand{\fx}{\mathfrak x}
\newcommand{\fv}{\mathfrak v}
\newcommand{\diam}{\mathrm{diam}}
\newcommand{\PgL}{\mathrm{P}\Gamma{\mathrm{L}}}
\newcommand{\LL}{\mathbb{L}}
\newcommand{\TT}{\mathbb{T}}
\newcommand{\QQ}{\mathbb{Q}}
\newcommand{\cG}{\mathcal{G}}
\newcommand{\ccQ}{\mathcal{Q}}
\newcommand{\RM}{\mathrm{RM}\,}
\newcommand{\trace}{\mbox{\itshape trace}}
\newcommand{\diag}{\mbox{\itshape diag}}
\newcommand{\spin}{\text{\itshape spin}}
\newcommand{\cha}{\text{\itshape char}}
\newcommand{\bE}{\mathbb E}
\newcommand{\fF}{\mathfrak F}
\newcommand{\fE}{\mathfrak E}
\newcommand{\fN}{\mathfrak N}
\newcommand{\fG}{\mathfrak G}
\newcommand{\baM}{\overline{M}}
\newcommand{\PG}{\mathrm{PG}}
\newcommand{\Sp}{\mathrm{Sp}}
\newcommand{\GL}{\mathrm{GL}}
\newcommand{\PGL}{\mathrm{PGL}}
\newcommand{\PGO}{\mathrm{PGO}}
\newcommand{\PGU}{\mathrm{PGU}}
\newcommand{\PSp}{\mathrm{PSp}}
\newcommand{\FF}{\mathbb{F}}
\newcommand{\ZZ}{\mathbb{Z}}
\newcommand{\NN}{\mathbb{N}}
\newcommand{\WW}{\mathbb{W}}
\newcommand{\bS}{\mathbb{S}}
\newcommand{\Rad}{\mathrm{Rad}}
\newcommand{\Res}{\mathrm{Res}}
\newcommand{\Fix}{\mathrm{Fix}}
\newcommand{\eps}{\varepsilon}
\newcommand{\Aut}{\mathrm{Aut}}
\newcommand{\lt}{\mathrm{lt}}
\newcommand{\gr}{\mathrm{gr}}
\newcommand{\er}{\mathrm{er}}
\newcommand{\aut}{\mathrm{Stab}}
\newcommand{\ch}{\mathrm{char}}
\newcommand{\rank}{\mathrm{rank}\,}
\newcommand{\N}{\mathcal{N}}
\newcommand{\ox}{\overline{x}}
\newcommand{\ov}{\overline{v}}
\newcommand{\oy}{\overline{y}}
\newcommand{\oU}{\widetilde{U}}
\newcommand{\oS}{\overline{S}}
\newcommand{\oM}{\overline{M}}
\newcommand{\ou}{\overline{u}}
\newcommand{\oV}{\overline{V}}
\newcommand{\oPi}{{\overline{\Pi}}_{\varphi}}
\newcommand{\oRad}{\overline{\mathrm{Rad}}(\varphi)}
\newcommand{\GF}{\mathrm{GF}}

\newcommand{\bZ}{\bf{0}}
\newcommand{\codim}{\mathrm{codim}\,}
\newcommand{\tV}{\tilde{V}}
\newcommand{\tve}{\tilde{\varepsilon}}
\theoremstyle{plain}

\newtheorem{lemma}{Lemma}[section]
\newtheorem{theorem}[lemma]{Theorem}
\newtheorem{corollary}[lemma]{Corollary}

\newtheorem{setting}[lemma]{Setting}
\theoremstyle{definition}
\newtheorem{definition}[lemma]{Definition}
\newtheorem{remark}[lemma]{Remark}

\def\pr{\noindent{\bf Proof. }}
\def\eop{\hspace*{\fill}$\Box$}

\newcommand{\fp}{\mathfrak p}
\newcommand{\fe}{\mathfrak e}
\newcommand{\fd}{\mathfrak d}
\newcommand{\fr}{\mathfrak r}

\begin{document}

\title{On orthogonal polar spaces}
\author{Ilaria Cardinali and  Luca Giuzzi }
\maketitle

\begin{abstract}
	Let $\cP$ be a non-degenerate polar space. In~\cite{ILP21b}, we introduced an intrinsic parameter of $\cP$, called the anisotropic gap, defined as the least upper bound of the lengths of the well-ordered chains of subspaces of $\cP$ containing a frame; when $\cP$ is orthogonal, we also defined two other parameters of $\cP$, called the elliptic and parabolic gap, both related to the universal embedding of	$\cP$.
	In this paper, assuming that $\cP$ is an orthogonal polar space, we prove that the elliptic and parabolic gaps can be described as intrinsic invariants	of $\cP$ without directly appealing to the embedding.
\end{abstract}
\leftline{{\bfseries MSC:} 51A50, 51B25, 51E24 }
\leftline{{\bfseries Keywords:} Polar Spaces, Quadratic forms, Embeddings, Hyperplanes.}

\section{Introduction}\label{introduction}
We assume that the reader is familiar with the concepts of polar spaces, projective embeddings and subspaces,  which we will briefly recall in Section~\ref{Prelim}.

Let $\cP:=(P,L)$ be a non-degenerate polar space, regarded as a point-line geometry.
We will say that any two points $p,q$ of $\cP$ are
	{\it collinear}, and write $p\perp q$, if there exists a line $\ell\in L$
incident with both $p$ and $q$.   For any $x\in P$ we shall write
$x^{\perp}$ for the set of all points of $\cP$ collinear with $x$; if $X\subseteq P, $ then $X^{\perp}:=\bigcap_{x\in X}x^{\perp}$.

A \emph{subspace} of $\cP$ is a subset $X\subseteq P$ such that every line containing at least two points of $X$ is entirely contained in $X$.
The intersection of all subspaces of $\cP$ containing a given subset
$S\subseteq P$ is a subspace called the \emph{span} of $S$ and denoted by $\langle S\rangle$. A {\em hyperplane} of $\cP$ is a proper subspace of $\cP$ meeting every line of $\cP$ non-trivially.
A subspace $X$ is {\it singular} if for all $x,y\in X$ we have $x\perp y$ (equivalently $X\subseteq X^{\perp}$).

The subspace $\cP^{\perp}: =\{ p\in P \colon p^\perp = P \}$ is called the \emph{radical} of $\cP$ and denoted by $\mathrm{Rad}(\cP)$. Obviously, $\mathrm{Rad}(\cP)$ is a singular subspace and it is contained in all maximal singular subspaces of $\cP$. The polar space $\cP$ is \emph{degenerate} precisely when $\mathrm{Rad}(\cP) \neq \emptyset$.

Every subspace $S$ of $\cP$ can be naturally endowed with the structure of a polar space by taking as points, the points of $S$ and as lines all lines of $\cP$ fully contained in $S$.
In particular, a singular subspace, regarded as a polar space, coincides with its own radical.

In this paper we shall always assume that $\cP$ is non-degenerate.
It is well-known that all singular subspaces of a non-degenerate polar
space are projective spaces; see~\cite[Theorem 7.3.6]{S}.
A hyperplane $\cH$ is called \emph{singular} if $\cH=p^\perp$ for $p\in \cP$.
If $\cP$ is non-degenerate, then any degenerate hyperplane $\cH$ of
$\cP$ is singular, i.e.\ $\cH=p^{\perp}$ where $\mathrm{Rad}(\cH)=\{p\}$.

The rank of a polar space $\cP$, usually denoted by $\rank(\cP)$, is
defined as the least upper bound of the lengths of the
well ordered chains of singular subspaces contained in it with $\emptyset$
regarded as the smallest singular subspace.
We recall that the \emph{length} of a
chain is its cardinality, diminished by $1$ when the chain is finite.
In particular, the rank of an infinite well ordered chain is just its  cardinality, namely the cardinality of the ordinal number representing
the isomorphism class of the chain itself.

The rank of a projective space is its generating rank, namely its dimension augmented by $1$. The \emph{rank} $\rank(X)$ of a singular subspace $X$ of $\cP$ is its rank as a projective space.

If the maximal singular subspaces of $\cP$ have all finite rank, then they all have the same rank. This common rank coincides with $\rank(\cP)$ according to the definition above.
For the sake of completeness, when $\cP$ admits singular subspaces of infinite rank we put $\rank(\cP) = \infty$ even if we shall not deal with this case in the present paper.
Polar spaces of rank $2$ are called
\emph{generalized quadrangles}. Polar spaces of rank $1$ are just sets
of pairwise non-collinear points.

Henceforth we shall always assume that $\cP$ has finite rank, say $n$.
Then, the hyperplanes of $\cP$ have rank either $n-1$ or $n$.

Since the rank of $\cP$ is finite, for any maximal singular subspace $M$ there exists another maximal singular subspace $M'$ disjoint from $M$.
Fixed a basis $(p_1,\dots, p_n)$ of $M$ (by {\it basis} of a projective space we mean a minimal generating set), there is a unique basis
$(p'_1,\dots, p'_n)$ of $M'$ such that $p_i\perp p'_j$ if and only if $i\neq j$. Such a pair $\{(p_1,\dots, p_n),(p'_1,\dots, p'_n)\}$ of generating
sets is called a {\em frame} of $\cP$ (see \cite{PasDG},~\cite{BC}).
In general, if $\cP$ is embeddable of rank $n$, a set spanning $\cP$ must
contain at least $2n$ points, the cardinality of a frame. In this case,
the frame is a minimal set spanning a subspace of $\cP$ with the same
rank as $\cP$.
Note that there are non-embeddable subspaces of rank $3$ which are spanned by a number of points smaller than the cardinality of a frame~\cite{P-VM}.

In~\cite{ILP21b}, we introduced an intrinsic parameter of $\cP$ called the \emph{anisotropic gap} (there under the name of {\it anisotropic defect}) of $\cP$  as follows; see also~\cite{Srv23}. Let $\mathfrak{N}(\cP)$ be the family, ordered by inclusion, of the well-ordered chains of subspaces of $\cP$ containing a frame.
The {\em anisotropic gap} $\mathrm{gap}({\cP})$ of $\cP$ is the least upper bound of the lengths of the elements of $\mathfrak{N}(\cP)$, i.e
\begin{equation}\label{anisotropic gap}
	\mathrm{gap}(\cP)=\sup\{|C|-1\colon C\in \mathfrak{N}(\cP)\}.
\end{equation}
In other words,
the anisotropic gap of $\cP$ tells us how ``far'' $\cP$ is from any of its subspaces spanned by frames.

A polar space is \emph{classical} if it is non-degenerate and
it admits the universal embedding; see Section~\ref{Prelim}.
Suppose that  $\cP$ is a classical polar space and let $\varepsilon:\cP\rightarrow\PG(V)$ be its universal embedding.
Call $\KK$ the underlying division ring of $V$.
In this case, $\varepsilon(\cP) = \cP(f)$ with $f$ a non-degenerate alternating form or $\varepsilon(\cP) = \cP(\phi)$ for a non-degenerate pseudoquadratic form $\phi$.
We recall that the {\it sesquilinearization} $f_\phi$ of $\phi$ (henceforth denoted by just $f$ for simplicity) can be degenerate only if $char(\KK)=2.$
When $\phi$ is a non-degenerate quadratic form, $\cP(\phi)$ is called a {\it non-degenerate orthogonal polar space}; in this case $\cP(\phi)$ is defined over a field.

Clearly, when considering two orthogonal vectors or subspaces of $V$, we refer to orthogonality with respect to $f$ or with respect to the sesquilinearization of $\phi$, according to the case.

Given a frame $A$ of $\cP$, its image $\varepsilon(A)$ spans a $2n$-dimensional subspace $H$ of $V$ which splits as the direct sum $V_1\oplus V_2\oplus\dots\oplus V_n$ of mutually orthogonal $2$-dimensional subspaces $V_1, V_2,\dots, V_n$. These subspaces bijectively correspond to the $n$ pairs of non-collinear points of $A$ and appear as lines in $\PG(V)$.
Denote by $[V_i]$ the projective line corresponding to the vector
subspace $V_i$ for $i=1,\dots,n$.
The preimages $\varepsilon^{-1}([V_1]),\dots, \varepsilon^{-1}([V_n])$ are hyperbolic lines of $\cP$ (a {\em hyperbolic line} of a polar space being the double perp ${\{p,q\}}^{\perp\perp}$ of two non-collinear points $p$ and $q$).
If $V_0$ denotes an orthogonal complement of $H=V_1\oplus\dots \oplus V_n$ in $V$, then $V_0$ is $\phi$-anisotropic, i.e.\ $\phi(x)\neq 0,\, \forall x\in V_0\setminus\{\mathbf{0}\}$.
If  $f$ is degenerate but $\cP$ is not,
then $V_0 \supseteq \mathrm{Rad}(f)$. In this case, suppose that $V_0'$ is a complement of $\mathrm{Rad}(f)$ in $V_0$, that is $V_0=V_0'\oplus \mathrm{Rad}(f)$.
We have the following orthogonal decomposition
\begin{equation}\label{decomposition}
	V ~ = ~ H\oplus V_0 ~ = ~  (V_1 \oplus V_2 \oplus \dots \oplus V_n)\oplus V_0 ~ = ~ (V_1 \oplus V_2 \oplus \dots \oplus V_n)\oplus V_0'\oplus \mathrm{Rad}(f).
\end{equation}

In~\cite{ILP21b}, we proved the following
\begin{theorem}\label{plain}
	Let $\cP$ be a classical non-degenerate polar space. Then the anisotropic gap of $\cP$ is precisely the codimension of $V_1 \oplus V_2 \oplus \dots \oplus V_n$ in $V$, i.e.\ $\mathrm{gap}(\cP) =\dim(V_0)$. Moreover, every well ordered chain of $\mathfrak{N}({\cP})$ is contained in a maximal well ordered chain and all maximal well ordered chains of $\mathfrak{N}({\cP})$ have the same length, namely $\mathrm{gap}(\cP)$.
\end{theorem}

In the same paper, for $\ch(\KK)=2$, we mentioned two more parameters called \emph{parabolic} and \emph{elliptic gaps}  of $\cP$ (there called parabolic and elliptic defects).

The parabolic gap corresponds to the dimension of the
radical $\mathrm{Rad}(f)$ and the elliptic gap is
defined as $\dim(V_0')$ (see~\eqref{decomposition}).  We point out that our definition of
\emph{parabolic gap} corresponds to the definition of \emph{corank}
of the form $\phi$ in~\cite{T} when the form $\phi$ is non-degenerate.

It turns out that the notion of parabolic gap is also closely
related to the existence of projective embeddings of a polar space
different from the universal one (see Section~\ref{section finale}).

In the present paper, assuming that $\cP$ admits the universal embedding, we provide an intrinsic characterization of
the notions of parabolic and elliptic gap of $\cP$ along the lines of Theorem~\ref{plain}, thus answering Problem 5.2  of~\cite{ILP21b}.

We will also give a characterization of orthogonal polar spaces
(see Definition~\ref{oeh})
without explicit reference to the quadratic form describing them.
Once more, we only require the existence of the universal embedding.

The main motivation of this paper is to continue on the project started in~\cite{ILP21b}, aimed to offer an embedding-free definition for
the notion of ``dimension'' of a classical polar space $\cP$.
The most natural way is to identify the dimension of $\cP$ with the vector dimension of its universal embedding, i.e.\ its embedding rank. Hence $\dim(\cP)=2n + \fd$, where $n$ is the rank of $\cP$ (for which we have an intrinsic definition) and $\fd$ is the dimension of a complement of the space spanned by a frame. In~\cite{ILP21b}  we proved that $\fd$ can be defined in an embedding-free way by the length of well-ordered chains of suitable subspaces of $\cP$, thus leading to the definition of anisotropic gap. Here, we continue the job focusing on orthogonal polar spaces in characteristic $2$, where $\fd$ is in turn the sum of two terms (the elliptic and the parabolic gap) which we characterize intrinsically respectively in Theorems~\ref{m:1} and~\ref{m:2}.
Our characterizations provide some insight on the
requirements for two orthogonal polar spaces to  be isomorphic and
provide the groundwork for some future research.
Indeed,
in many cases, even in characteristic $0$, it is not sufficient for two
orthogonal polar spaces to have both the same rank and the same
anisotropic gap in order to
be isomorphic; for instance there are non-isomorphic orthogonal
polar spaces over the rational field $\mathbb Q$ which have
the same parameters; on the other hand, two orthogonal
polar spaces with the same rank and anisotropic gap over the real field $\mathbb R$ are isomorphic.

We aim to further investigate the relationship between
the fields involved, the gaps we defined, and isomorphism classes in
future works.
Furthermore, the notions we introduce in the present paper
can be extended also to polar spaces described by
hermitian  forms in characteristic $2$ (over non-commutative
division rings) as well as to some non-embeddable cases. We leave this
study to a future work.

We warn the reader that the results of the present paper all
encompass
the case in which the gaps are possibly infinite,
where a Witt-like decomposition of the quadratic form describing
the universal embedding might not be easily manageable.
None the less, as
mentioned above, we restrict our current analysis to spaces of finite
rank $n$.

Throughout the paper we use greek letters to denote ordinals related to
chains and blackletter characters for cardinal numbers which might possibly
be infinite.

\medskip
We will say that an embedded non-degenerate orthogonal polar space $\cP(\phi)$ of rank $n$ is \emph{$(\fe,\fp)$-orthogonal} if in the decomposition~\eqref{decomposition}
$\dim(V_0')=\fe$ and $\dim(\mathrm{Rad}(f))=\fp$.
More in particular, we say that $\cP(\phi)$ is
of \emph{hyperbolic type} if $\fe=\fp=0$; so
$\cP(\phi)$ is spanned by a frame and
\[V ~ = V_1 \oplus V_2 \oplus \dots \oplus V_n;\]
it  is \emph{elliptic} if $\fe>0$ but $\fp=0$, that is
$\cP(\phi)$ is not hyperbolic and the bilinear form $f$ polarizing $\phi$ is non-degenerate, i.e.\ $\mathrm{Rad}(f)=\{\mathbf{0}\}$ and $V_0=V_0'$
in~\eqref{decomposition}:
\[V ~ = (V_1 \oplus V_2 \oplus \dots \oplus V_n)\oplus V_0';\]
it is \emph{parabolic} if $\fe=0$ and $\fp>0$, that is it is not hyperbolic and $\mathrm{Rad}(f)=V_0$, i.e.\ $\mathrm{Rad}(f)$ is a non-trivial direct complement of
$H=\bigoplus_{i=1}^{n}V_i$ in $V$, i.e.
\[V ~ = (V_1 \oplus V_2 \oplus \dots \oplus V_n)\oplus \mathrm{Rad}(f).\]
The aim of this paper is to provide an
intrinsic description of $(\fe,\fp)$-orthogonal spaces without mentioning
the embedding; for this, we refer to  Corollary~\ref{d112}.
We observe that in general there
exist $(\fe,\fp)$-orthogonal spaces where both $\fe>0$ and $\fp>0$.

The following results, to be proved in Section~\ref{hyperb lines}, characterize hyperbolic and elliptic orthogonal polar spaces relying on the cardinality of their hyperbolic lines.
\begin{theorem}\label{thm hyperb lines}
	Let $\cP$ be an embeddable polar space  and $\varepsilon$ a relatively universal embedding of $\cP$. The embedded polar space
	$\varepsilon(\cP) $ is $(\fe,0)$-orthogonal (i.e.\ either of hyperbolic or
	elliptic type) if and only if the hyperbolic lines of $\cP$ contain exactly $2$ points.
\end{theorem}
In light of Theorem~\ref{thm hyperb lines}, it is possible to provide
an intrinsic characterization encompassing all orthogonal polar spaces,
as shown by the following corollary.
\begin{corollary}
	\label{pr:orth}
	Let $\cP$ be an embeddable polar space and $\varepsilon$ be
        a relatively universal embedding of $\cP$. The embedded polar space   $\varepsilon(\cP)$ is orthogonal i.e.\ $\varepsilon(\cP)=\cP(\phi)$ where
	$\phi$ is a quadratic form, if and only if any subspace  $\cF$ of $\cP$
	generated by a frame has the property that all its hyperbolic lines consist of exactly $2$ points.
\end{corollary}
Using these results we can formulate the following definition of
orthogonal polar spaces which does not explicitly mention the
universal embedding.
\begin{definition}
	\label{oeh}
	A non-degenerate
	embeddable polar space $\cP$ is \emph{orthogonal} if the hyperbolic lines
	of any subspace $\cF$ of $\cP$ generated by a frame contain exactly $2$ points.
	A non-degenerate polar (sub)space $\cP$
	is \emph{hyperbolic} if it is orthogonal and generated by a
	frame;
	it is \emph{elliptic} if it is orthogonal but
	not hyperbolic and each of its hyperbolic lines consists of $2$
	points.
\end{definition}
We warn the reader that,
according to this definition, if $\mathrm{char}(\KK)\neq 2$, then the quadrics
of rank $n$ in dimension $2n+1$, usually called parabolic, will be here
named elliptic; see Remark~\ref{rem13} for more details on this choice.
\begin{definition}
  \label{dec}
  Suppose that $\cP$ is an orthogonal non-degenerate polar space.
  We say that a well ordered chain of subspaces
  \begin{equation}\label{elliptic chain}
		\fE: \cF=\cE_0\subset\cE_1\subset\dots\subset\cE_\mu\subset\cdots \subset \cP
	\end{equation}
	is an \emph{elliptic chain} of $\cP$ if any $\cE_i$ for $i\geq1$
	is  elliptic and $\cF$ is a hyperbolic subspace of $\cP$.
\end{definition}

\begin{remark}
	It might be possible to extend Definition~\ref{dec} of \emph{elliptic subspace} and
	\emph{elliptic chain} also to polar spaces which are not orthogonal, by stating that
	a subspace $\cE$ of $\cP$ containing a frame $F$ of $\cP$ is ``\emph{elliptic}'' if, called $\cF$ the subspace of $\cE$
	generated by $F$, we have $\cF\neq\cE$ and $\forall p,q\in\cF$ with
	$p\not\perp q$
	\[ {\{p,q\}}^{\perp\perp}\cap\cF={\{p,q\}}^{\perp\perp}\cap\cE, \]
	that is the hyperbolic lines of $\cF$ are also hyperbolic lines of $\cE$.
	As the group of a classical polar space of finite rank is transitive on the hyperbolic lines,
	this new definition for orthogonal polar spaces is equivalent to Definition~\ref{oeh}.
	On the other hand, non-elliptic subspaces might occur for non-orthogonal
	classical polar spaces $\cP$ of finite rank only if $\KK$ is a
	non-commutative division ring in characteristic $2$.
        We leave considering these cases to a future work.
\end{remark}

Our main result is the following theorem.
\begin{theorem}
	\label{m:1}
	Let $\cP$ be a non-degenerate orthogonal polar space of rank at least $2$ embeddable over a field $\KK$. Then all well ordered maximal elliptic chains of $\cP$ have the same cardinality $\fd$. In particular, if $\mathrm{char}(\KK)\neq 2$, then
	$\fd$ is precisely the anisotropic  gap of $\cP$ while if $\mathrm{char}(\KK)=2$, then all maximal
	subspaces of $\cP$ having the property that all their hyperbolic lines contain exactly $2$ points
	have the same
	anisotropic gap equal to $2\fd$ and $2\fd$ is precisely the codimension of $\mathrm{Rad}(f)$ in $V_0$, where $f$ is the bilinearization of a (relatively) universal embedding of $\cP$ and $V_0$
        is as in~\eqref{decomposition}.
\end{theorem}
We prove in Corollary~\ref{e:max} that any elliptic chain $\fE$
admits a maximal element $\cE_{\omega}$ which is an elliptic subspace.
In light of this and Theorem~\ref{m:1}, the \emph{elliptic gap}
of $\cP$, already introduced before, can be defined as follows.

\begin{definition}\label{elliptic gap}
	Let $\cP$ be an
	orthogonal polar space of rank at least $2$ and
	$\fE: \cF=\cE_0\subset\cE_1\subset\dots\subset\cE_{\omega}$
	be a maximal well-ordered
	elliptic chain of $\cP$ where $\cF$ is the subspace spanned by a frame
	and $\cE_{\omega}$ is a maximal element of $\fE$.
	A \emph{maximal enrichment} of $\fE$ is a (possibly non-elliptic)
	well ordered chain
	$\fE^E: \cX_0=\cE_0\subseteq\cX_1\subseteq\dots\subseteq\cX_{\theta}=\cE_{\omega}$,
	containing all elements of $\fE$,
	starting at $\cE_0$ and ending at $\cE_{\omega}$ and maximal with
	respect to these properties.
	
	The \emph{elliptic gap} of $\cP$ is the length of
	a maximal enrichment $\fE^E$ of a maximal elliptic chain $\fE$ of $\cP$.
\end{definition}
Observe that for any $0\leq i<\theta$, the space $\cX_{i}$ in $\fE^E$
is necessarily a maximal subspace of $\cX_{i+1}$.

Note also
that if $\mathrm{char}(\KK)\not=2$, then $\fE\equiv \fE^E$, i.e.\ a maximal
elliptic chain $\fE$ admits no proper enrichment, namely it is maximal as a chain of subspaces containing the span of a given frame; also,
in this case, the elliptic gap coincides with the anisotropic gap of $\cP.$
In contrast, if $\mathrm{char}(\KK)=2$, then $\fE$ always admits proper refinements and
in this case the elliptic gap of $\cP$ is twice the length of $\fE$ (of course this makes sense only if the elliptic gap is finite).
The anisotropic gap of $\cP$ is possibly larger.
In any case, the elliptic gap of $\cP$ is the anisotropic gap of
a maximal subspace of $\cP$ with the property that all of its hyperbolic lines contain exactly $2$ points, i.e.\ it is the anisotropic gap of a maximal elliptic subspace of $\cP$.

A consequence of Theorem~\ref{m:1} is the following.
\begin{theorem}
	\label{m:2}
	Let $\cP$ be a non-degenerate orthogonal
	polar space of rank at least $2$ and $\varepsilon\colon \cP\rightarrow \PG(V)$ be its universal embedding.
	Let $\fE^E$ be an enrichment of a maximal elliptic chain $\fE$ of $\cP$.
	Then
	$\dim({\mathrm{Rad}}(f))=\dim\left(V/\left\langle\varepsilon\left(
			\bigcup_{X\in \fE^E}X\right)\right\rangle\right)$.
\end{theorem}
\begin{remark}
	\label{r:111}
	If we call $\fr$ the anisotropic gap of $\cP$ and
	$\fd$ its elliptic gap and both $\fr$ and $\fd$ are
	finite, then
	the statement of Theorem~\ref{m:2} reads as
	$\dim(\mathrm{Rad}(f))=\fr-\fd$.
	When $\fr$ and $\fd$ are possibly infinite cardinals
	the statement
	should be read as follows: denote by $\fr$ the anisotropic gap
	of $\cP$ and
	let $\cE$ be a maximal elliptic subspace of $\cP$ of
	anisotropic gap $\fd$; then there exists a maximal chain ${\mathfrak M}$
	of
	subspaces of $\cP$ containing a frame of $\cP$ as well as $\cE$
	and this chain has length $\fr$.
	Let ${\mathfrak z}$ be the length
	of the subchain of $\mathfrak M$ from $\cE$ to $\cP$; then
	${\mathfrak z}+\fd=\fr$.
\end{remark}
\begin{definition}
	By Theorem~\ref{m:2}, the
	\emph{parabolic gap} of $\cP$ can be defined as the
	cardinality of any maximal well ordered chain $\mathfrak M$
	of subspaces of $\cP$
	all containing a space $\cE:=\bigcup_{X\in \fE}(X)$, where $\fE$ is any
	maximal well ordered elliptic chain of $\cP$.
\end{definition}

In light of what we have said so far, the following corollary
shows that the notion of $(\fe,\fp)$-orthogonal space can be
formulated without any explicit
mention of its embeddings and, as such, is intrinsic to the space.

\begin{corollary}
	\label{d112}
	Let $\cP$ be a non-degenerate embeddable
	polar space such that the hyperbolic lines
	of a subspace $\cF$
	generated by a frame of $\cP$ consist of only two points.
	Then $\varepsilon(\cP)$ is an $(\fe,\fp)$-orthogonal polar space
	where $e$ and $p$ are respectively the elliptic and the parabolic
	gaps of $\cP$.
	In particular, if $\fe=\fp=0$, the space $\cP$ is hyperbolic;
	if $\fe>0$ and $\fp=0$, then $\cP$ is elliptic and if $\fe=0$ and $\fp>0$
	the space is parabolic.
\end{corollary}

\begin{remark}
	\label{rem13}
	The terminology we have chosen for elliptic and  parabolic polar spaces is motivated by the terminology in use for quadrics over finite fields of characteristic $2$. Indeed, if $\KK=\mathrm{GF}(2^r)$ and $\cQ$ is a non-degenerate orthogonal polar space in $\PG(V)$, $V=V(n,\KK)$, described by a quadratic form $\phi$ with bilinearization $f$, then we have only two possibilities for $\dim(\mathrm{Rad}(f))$, i.e.\ $\dim(\mathrm{Rad}(f))=0$ or $\dim(\mathrm{Rad}(f))=1$.  In this case, the anisotropic space $V_0$ has dimension at most $2$. If $\dim(\mathrm{Rad}(f))=0$ and $V_0=0$, then $\cQ$ is usually called {\it hyperbolic quadric}; if $\dim(\mathrm{Rad}(f))=0$ and $\dim(V_0)=2$, then $\cQ$ is usually called {\it elliptic quadric}; if $\dim(\mathrm{Rad}(f))=1$ and $V_0=\mathrm{Rad}(f)$, then $\cQ$ is usually called {\it parabolic quadric}.
	
	When $\mathrm{char}(\KK)\neq 2$ our terminology is not the same as
	what is usually found in literature for quadrics.
	In our case, \emph{all} orthogonal polar spaces are either
	hyperbolic (i.e.\ generated by a frame) or elliptic and
	their anisotropic gap is the same as their elliptic gap.
		
	It is worth to mention that if $\KK$ is a non-perfect field of characteristic $2$ (i.e.\ $\KK$ has characteristic $2$ but not all elements of $\KK$ are squares in $\KK$), then there can exist non-degenerate $(\fe,\fp)$-orthogonal polar spaces which are neither hyperbolic, nor  elliptic nor parabolic, i.e.\ for them  $0\subset\mathrm{Rad}(f)\subset V_0$ and both $\fe,\fp>0$.
	
	Note that if $\cQ$ has $\dim(\mathrm{Rad}(f))=[\KK:\KK^2]<\infty$, where  $\KK^2=\{a^2\colon a\in \KK\}$, then $\cQ$ is parabolic, see Corollary~\ref{c:cpp}, but the converse is not true, i.e.\ there exist parabolic spaces (i.e.\ such that $\mathrm{Rad}(f)=V_0$) with $\dim(\mathrm{Rad}(f))<[\KK:\KK^2]$.
\end{remark}

\paragraph{Structure of the paper}
In Section~\ref{Prelim} we will give some basics on  embeddable polar spaces recalling the definitions and the relevant results from the literature. We will also set the notation.
In Section~\ref{hyperb lines} we shall focus on hyperbolic lines of $\cP$ thus proving Theorem~\ref{thm hyperb lines}. In Section~\ref{sec 3} we shall characterize the elliptic gap and we will prove Theorem~\ref{m:1}.
In Section~\ref{parabolic gap} we shall focus on the parabolic gap proving Theorem~\ref{m:2} and we shall characterize the  parabolic polar spaces.

\section{Preliminaries}\label{Prelim}

Let $\cP =(P, L)$ be a non-degenerate polar space and let $V$ be a vector space
over some division ring $\KK$.
A \emph{projective embedding} of $\cP$ is an injective map
$\varepsilon:\cP\rightarrow \PG(V)$ with the property that $\langle \eps(P)\rangle=\PG(V)$
and every line of $\cP$ is mapped onto a projective line of $\PG(V)$. So, $\varepsilon(L):={\{\varepsilon(\ell)\}}_{\ell\in L}$ is a set of lines of $\PG(V)$ and $\varepsilon(\cP) := (\varepsilon(P), \varepsilon(L))$ is a full subgeometry of (the point-line space of) $\PG(V)$.

If  $\eps_1:\cP \rightarrow \PG(V_1)$ and $\eps_2:\cP \rightarrow \PG(V_2)$ are two projective embeddings of $\cP$
we say that  $\eps_1$ \emph{covers} $\eps_2$ (in symbols, $\eps_2\leq \eps_1$) if  $V_1$ and $V_2$ are defined over the same division ring $\KK$ and there exists a $\KK$-semilinear mapping $\pi\colon V_1\rightarrow V_2$ such that $\eps_2= \pi\circ \eps_1$ (actually, writing $\pi\circ\eps_1$ is an abuse, since morphisms of projective spaces are involved here rather than their underlying semilinear maps, but this is a harmless abuse).
The map $\pi$ such that $\eps_2=\pi\circ\eps_1$, if it exists, is unique up to
rescaling. It is called the projection of $\eps_1$ onto $\eps_2$.
We say that $\eps_1$ and $\eps_2$ are
\emph{equivalent} (in symbols $\eps_1\simeq\eps_2$) if $\eps_1\leq \eps_2\leq \eps_1$; this is the same as to say that the projection $\pi:V_1\to V_2$ of $\eps_1$ onto $\eps_2$ is an isomorphism.

An embedding $\tilde{\varepsilon}$ is said to be \emph{relatively universal}  if $\bar{\varepsilon}\geq\tilde{\varepsilon}$ implies
$\bar{\varepsilon}\simeq\tilde{\varepsilon}$.
Note that every embedding $\varepsilon$ is covered by a relatively universal embedding (Ronan~\cite{Ron}), unique modulo equivalence. This relatively universal embedding is called the \emph{hull} of $\varepsilon$.

An embedding $\tilde{\varepsilon}$ of $\cP$ is \emph{absolutely universal} if it covers all projective embeddings of $\cP$. Clearly, when it exists, the universal embedding is unique up to equivalence and it is the hull of all embeddings of $\cP$. In this case, all embeddings of $\cP$ are necessarily defined over the same division ring.

The \emph{embedding rank} $\er(\cP)$ of $\cP$ is the least upper bound of the dimensions of the embeddings of $\cP$ (the maximal dimension of an embedding of $\cP$ when all these dimensions are finite and range in a finite set). Suppose that $\cP$ admits the absolutely universal embedding, say $\tilde{\varepsilon}$. Then $\dim(\tilde{\varepsilon}) =  \er(\cP)$ and, for any embedding $\varepsilon$ of $\cP$, if $\er(\cP)  < \infty$, then $\varepsilon\simeq\tilde{\varepsilon}$ if and only if $\dim({\varepsilon})=\er(\cP)$.

As proved by Tits~\cite[chp. 7-9]{T} (also Buekenhout and Cohen~\cite[chp. 7-11]{BC}), all thick-lined non-degenerate polar spaces of rank at least $3$ are embeddable except  for the following two exceptions, both of rank $3$: the line-grassmannian of $\PG(3,\KK)$ with $\KK$ a non-commutative division ring and a family of polar spaces of rank $3$ with non-desarguesian planes, described in~\cite[chp. 9]{T} (also Freudenthal~\cite{Fr}) which we call {\em Freudenthal-Tits polar spaces}; see also~\cite{P-VM} for this geometry.

Suppose now that $\cP$ is an embeddable non-degenerate polar space of finite rank $n \geq 2$. By Tits~\cite[chp. 8]{T}, the polar space $\cP$ admits the universal embedding but for the following two exceptions of rank $2$:
\begin{enumerate}[(E1)]
	\item\label{E1} $\cP$ is a grid with lines of size $s+1 > 4$, where $s$ is a prime power if $s < \infty$. If $\varepsilon:\cP\rightarrow\PG(V)$ is an embedding of $\cP$, then $V \cong V(4,\KK)$ for a field $\KK$ and $\varepsilon(\cP)$ is a hyperbolic quadric of $\PG(V)$. The field $\KK$ is uniquely determined by $\cP$ only if $s < \infty$.
	\item\label{E2}  $\cP$ is a generalized quadrangle admitting just two non-isomorphic embeddings $\varepsilon_1, \varepsilon_2:\cP\rightarrow \PG(3,\KK)$ for a quaternion division ring $\KK$ (the same for $\varepsilon_1$ and $\varepsilon_2$). We refer to~\cite[8.6]{T} for more on these examples
	which  we call \emph{bi-embeddable quaternion quadrangles}.
\end{enumerate}
Many polar spaces admit just a unique embedding, which is the absolutely universal one. The grids of case (E\ref{E1}) admit several
non-equivalent embeddings which are all relatively universal and
each of them is described by a quadratic form; as grids are generated
by frames, these turn out to be $(0,0)$-orthogonal polar spaces.

The bi-embeddable quaternion quadrangles of case (E\ref{E2}) are
generated by frames but their hyperbolic lines contain more than $2$
points, so they do not play any further role in our study here.

We recall that in characteristic $2$ there are more possibilities
for the embeddings of polar spaces (see Section~\ref{sesquilinear and pseudoquadratic form}).

\subsection{Sesquilinear and pseudoquadratic forms}\label{sesquilinear and pseudoquadratic form}
\label{s2}
Sesquilinear and pseudoquadratic forms are essential in order to describe
the projective embeddings of polar spaces. In particular, the universal embedding of a polar space
(when it exists) can always be described by either a sesquilinear or a quadratic form.

Given an anti-automorphism $\sigma$ of $\KK$, a $\sigma$-{\em sesquilinear form} is a function $f:V\times V\to\KK$ such that $f(\sum_ix_it_i, \sum_jy_js_j) = \sum_{i,j}t_i^\sigma f(x_i, y_j)s_j$ for any choice of vectors $x_i, y_j \in V$ and scalars $t_i, s_j \in \KK$. A sesquilinear form $f$ is said to be {\em reflexive} when for any two vectors $x, y\in V$ we have $f(x,y) = 0$ if and only if $f(y,x) = 0$.

Let $f$ be a non-trivial (i.e.\ not null) $\sigma$-sesquilinear form.
Then $f$ is reflexive if and only if there exists a scalar $\epsilon\in \KK^\ast$ such that $f(y,x) = {f(x,y)}^\sigma\epsilon$ for any choice of $x, y\in V$; this condition forces $\epsilon^\sigma = \epsilon^{-1}$ and $t^{\sigma^2} = \epsilon t\epsilon^{-1}$ for any $t\in \KK$ (see Bourbaki~\cite{Bour}). With $\epsilon$ as above, $f$ is called a $(\sigma,\epsilon)$-{\em sesquilinear form}. Clearly, $\epsilon \in \{1, -1\}$ if and only if $\sigma^2 = \mathrm{id}_\KK$; also, $\sigma = \mathrm{id_\KK}$ only if $\KK$ is commutative.
Let $\sigma= \mathrm{id}_\KK$. If $\epsilon = 1$, then $f$ is said to be {\em symmetric}; when $\mathrm{char}(\KK) \neq 2$, then $\epsilon = -1$ if and only if $f(x,x) = 0$ for any $x\in V$. In this case $f$ is said to be {\em alternating}. When $\mathrm{char}(\KK) = 2$, an {\em alternating} form is a $(\mathrm{id}_{\KK}, 1)$-form $f$ such that $f(x,x) = 0$ for every $x\in V$. A $(\sigma, \epsilon)$-form with $\sigma\neq\mathrm{id}_\KK$ and $\epsilon = 1$ (or $\epsilon = -1$) is called {\em hermitian} (respectively {\em antihermitian}).

Two vectors $v, w\in V$ are {\em orthogonal} with respect to $f$ (in symbols $v\perp_f w$) if $f(v,w) = 0$. A vector $v\in V$ is {\em isotropic} if $v\perp_f v$. A subspace $X$ of $V$ is {\em totally isotropic} if $X\subset X^{\perp_f}$. In contrast, if $\mathbf{0}$ is the unique isotropic vector of a subspace $X$ of $V$, then we say that $f$ is {\em anisotropic} over $X$ and that $X$ is anisotropic for $f$. The same terminology is adopted for points and subspaces of $\PG(V)$, in an obvious way. The subspace $\mathrm{Rad}(f) = V^{\perp_f}$ is the {\em radical} of $f$. The form $f$ is {\em degenerate} if $\mathrm{Rad}(f) \neq \{\mathbf{0}\}$.

The isotropic points of $\PG(V)$ together with the totally isotropic lines of $\PG(V)$ form a polar space $\cP(f)$. The singular subspaces of $\cP(f)$ are precisely the totally isotropic subspaces of $\PG(V)$ and the radical of $\cP(f)$ is the (subspace of $\PG(V)$ corresponding to) $\mathrm{Rad}(f)$. In particular, $\cP(f)$ is non-degenerate if and only if $f$ is non-degenerate.

If we are interested in the polar space $\cP(f)$ rather than in peculiar properties of the underlying form $f$, then we can always assume that $f$ is either alternating, symmetric or hermitian (or antihermitian, if we prefer).

Let us turn now to pseudoquadratic forms.
Let $\sigma$ and $\epsilon$ be as above, but with $\epsilon\neq-1$ when $\sigma = \mathrm{id}_\KK$ and $\mathrm{char}(\KK)\neq 2$. Put
$\KK_{\sigma,\epsilon}:= \{t-t^{\sigma}\epsilon\colon t\in\KK\}$. The set $\KK_{\sigma, \epsilon}$ is a subgroup of the additive group of $\KK$. Moreover $\KK_{\sigma,\epsilon}\subset \KK$, in view of the hypotheses we have assumed on $\sigma$ and $\epsilon$. Hence the quotient group $\KK/\KK_{\sigma,\epsilon}$ is non-trivial.

A function $\phi:V\to\KK/\KK_{\sigma,\epsilon}$ is a \emph{$(\sigma,\varepsilon)$-pseudoquadratic
	form} if there exists a $\sigma$-sesquilinear form $g:V\times V\to\KK$ such that $\phi(x)=g(x,x)+\KK_{\sigma,\varepsilon}$ for
all $x\in V$.

The sesquilinear form $g$ is not uniquely determined by $\phi$ and needs not
to be reflexive. In contrast, the form $f_\phi:V\times V\rightarrow \KK$ defined by the clause $f_\phi(x,y) := g(x,y) + {g(y,x)}^\sigma\epsilon$ only depends on $\phi$ and is $(\sigma,\epsilon)$-sesquilinear (hence reflexive). Moreover $\phi(x+y) = \phi(x) + \phi(y) + (f_\phi(x,y) + \KK_{\sigma,\epsilon})$ for any choice of $x, y \in V$. We call $f_\phi$ the {\em sesquilinearization} of $\phi$.

In general, there is no way to recover $\phi$ from its sesquilinearization $f_\phi$. However, suppose that the center of $\KK$ contains an element $\mu$ such that $\mu+\mu^{\sigma}\neq 0$ (as it is always the case when $\mathrm{char}(\KK) \neq 2$: just choose $\mu = 1$). Then
\[\phi(x) ~ = ~ \frac{\mu}{\mu+\mu^{\sigma}}\cdot f_\phi(x,x) + \KK_{\sigma,\epsilon}.\]
A vector $v\in V$ is \emph{singular} for $\phi$ if $\phi(v) = 0$. A subspace $X$ of $V$ is \emph{totally singular} if all of its nonzero vectors are singular while, if $\mathbf{0}$ is the unique singular vector of $X$, then we say that $\phi$ is {\em anisotropic} over $X$. The same terminology is used for points and subspaces of $\PG(V)$. The singular points of $\PG(V)$ together with the totally singular lines of $\PG(V)$ form a polar space $\cP(\phi)$. The singular subspaces of $\cP(\phi)$ are precisely the totally singular subspaces of $\PG(V)$. If the singular points of $\phi$ span $V$ (as it is always the case when $\phi$ admits at least one singular vector not in $\mathrm{Rad}(\phi)$), then the inclusion mapping of $\cP(\phi)$ in $\PG(V)$ is a projective embedding.

The set of singular vectors of $\mathrm{Rad}(f_\phi)$ is a subspace of $V$. It is called the {\em radical} of $\phi$ and denoted by $\mathrm{Rad}(\phi)$. The form $\phi$ is {\em non-degenerate} if $\mathrm{Rad}(\phi) = \{\mathbf{0}\}$, namely $\phi$ is anisotropic over $\mathrm{Rad}(f_\phi)$. The radical of the polar space $\cP(\phi)$ is (the subspace of $\PG(V)$ corresponding to) $\mathrm{Rad}(\phi)$. So, $\cP(\phi)$ is non-degenerate if and only if $\phi$ is non-degenerate.

All vectors that are singular for $\phi$ are isotropic for $f_\phi$ and the span $\langle v, w\rangle$ of two singular vectors $v, w\in V$ is totally singular if and only if $v\perp_{f_\phi} w$. It follows that $\cP(\phi)$ is a subspace of $\cP(f_\phi)$, but possibly different from $\cP(f_\phi)$. In particular, it can  happen that $\cP(\phi)$ is non-degenerate and $\cP(f_\phi)$ is degenerate. However, when $\phi$ can be recovered from $f_\phi$, then $\cP(\phi) = \cP(f_\phi)$. If this is the case, then all we can do with $\phi$ can be done with $f_\phi$ as well.

When focusing on $\cP(\phi)$ regardless of peculiar properties of the form $\phi$, we can always assume that $\phi$ is $(\sigma,1)$-pseudoquadratic.
A $(\sigma,1)$-pseudoquadratic form is called \emph{quadratic} or
\emph{hermitian} according to whether $\sigma=\mathrm{id}_{\KK}$ or
$\sigma\neq\mathrm{id}_{\KK}$.
\medskip

Let now $\varepsilon:{\cP}\rightarrow \PG(V)$ be the universal embedding of $\cP$, where $V = V(N,\KK)$  and $\KK$ is  a division ring.
By Tits~\cite[Chapter 8]{T}, $\varepsilon(\cP)=\cP(f)$ for a non-degenerate sesquilinear form $f$ or $\varepsilon(\cP)=\cP(\phi)$ for a non-degenerate pseudoquadratic form $\phi$.

By Tits~\cite[Chapter 8]{T}, since $\varepsilon$ is universal, non-degenerate and trace-valued, $f$ cannot be alternating when $\ch(\KK) = 2$.
If either $\ch(\KK)\neq 2$ or $\ch(\KK)=2$ but $\sigma$ acts non-trivially
on the center of $\KK$ (more generally, $\phi$ is uniquely determined by
$f_{\phi}$), then $\varepsilon$ is the unique embedding of $\cP$.

If $\ch(\KK) = 2$ and $\eps(\cP)=\cP(\phi)$, then $R := \mathrm{Rad}(f_\phi)$ may be non-trivial. Suppose $R \neq \{\mathbf{0}\}$. However no nonzero vector of $R$ is singular for $\phi$, since $\phi$ is non-degenerate. Hence $R$ contains no point of $\cP(\phi)$. Moreover, every projective line of $\PG(V)$ containing at least two $\phi$-singular points misses $R$. Therefore, for every subspace $X$ of $R$, the projection $\pi_X:\PG(V)\rightarrow \PG(V/X)$ induces an injective mapping from $\cP(\phi)$. Accordingly, the composite $\varepsilon_X = \pi_X\circ\eps$ is an embedding of $\cP$. All embeddings of $\cP$ arise in this way, by factorizing $V$ over subspaces of $R$. The embedding $\varepsilon_R$ obtained by factorizing $V$ over $R$ is the minimal one.

All embeddings $\varepsilon_X$ (included the ones above where $\KK_{\sigma,1}=\{0\}$) can be described by means of {\em generalized pseudoquadratic forms} (see~\cite{P}), except the minimal one when $\phi(R) = \KK/\KK_{\sigma,1}$. If $\phi(R) = \KK/\KK_{\sigma,1}$, then $\sigma = \mathrm{id}_\KK$, the form $\phi$ is quadratic and it induces the null form on $V' := V/R$ and $\varepsilon_R(\cP) = \cP(f_\phi)$ for a non-degenerate alternating form $f:V'\times V'\rightarrow \KK$ (hence $\dim(V') = 2n$). Accordingly, if $\ch(\KK) = 2$ and $\cP$ admits an embedding $\eps':\cP\rightarrow\PG(V')$ such that $\eps'(\cP)$ is a symplectic variety of $\PG(V')$, then the universal embedding $\eps:\cP\rightarrow\PG(V)$ embeds $\cP$ as a quadric $\cP(\phi)$ of $\PG(V)$ and $V' = V/R$. Moreover, if
$\cP$ has finite rank $n$, then $\dim(\eps) = 2n+\fd$ ($= \fd$ when $\fd$ is infinite), where $\fd$ is the degree of $\KK$ over its subfield $\KK^2 = {\{t^2\}}_{t\in \KK}$ (see~\cite{P}). In particular, if $\KK$ is perfect, then
$\fd=1$.

\subsection{Subspaces of a polar space}
Suppose that $\cP$ is an embeddable non-degenerate polar space of finite rank $n\geq 2$ and let  $\varepsilon:{\cP}\rightarrow \PG(V)$ be an embedding of $\cP$.
For any $Z\leq V$ we shall denote by $[Z]$ the projective subspace
of $\PG(V)$ with underlying vector space $Z$.
If $\cS$ is a subspace of $\cP$, then we say that $\cS$ \emph{arises} from $\eps$ if $\cS=\varepsilon^{-1}([X])$, where $[X]$ is a (projective) subspace of $\PG(V)$.
In~\cite{ILP21a} we proved the following (see also~\cite[Remark 2.1]{P22}
for case (E\ref{E2})):
\begin{theorem}\label{thm subspaces}
	All subspaces of non-degenerate rank at least $2$ of an embeddable polar space
	of finite rank $n\geq 2$ arise from at least one of its embedding
	(hence
	from its universal embedding, if the polar space is neither a grid as in
	(E\ref{E1}) nor a bi-embeddable quaternion quadrangle as in (E\ref{E2})).
\end{theorem}
When $\cP$ is a grid or a bi-embeddable quadrangle (mentioned at the beginning of Section~\ref{sesquilinear and pseudoquadratic form})
there
are no proper subspaces of non-degenerate rank $2$, so the statement of the
theorem is empty, hence trivially true, in those cases.

In Section 1, we already defined a frame of $\cP$ as a pair $F=\{A,B\}$, where $A$ and $B$ are mutually disjoint sets of points of $\cP$, both of size $n$, such that $A\subseteq A^\perp$, $B\subseteq B^\perp$ and
for any $a\in A$ there exists a unique $b\in B$ such that $a\not\perp b$,
and conversely for any $b\in B$ there exists a unique $a\in A$ with
$b\not\perp a$.
We say that a subspace of $\cP$ is \emph{nice} if it contains a frame of
$\cP$.
In particular, any subspace of $\cP$ containing a frame
is non-degenerate and
has the same rank as $\cP$.
The following property is immediate yet useful.
\begin{remark}
	\label{nice hyperplane}
	No frame of $\cP$ is contained in the perp of a point; consequently,
	no nice subspace is contained in a singular hyperplane of $\cP$.
\end{remark}

Let $N({\cP})$ be the poset of the nice subspaces of $\cP$
ordered by inclusion.
Following
the notation of Section 1, we shall denote by $\fN(\cP)$
the family of all well-ordered chains of elements of $N(\cP)$ again
ordered by inclusion.
Clearly, $\cP$ is the greatest element of $N({\cal P})$ and the minimal elements of ${N}({\cal P})$ are exactly the subspaces spanned
by the frames of $\cP$.
The poset $N(\cP)$ always contains finite chains and, trivially, every finite chain is well ordered; so $\fN(\cP)\neq\emptyset$.
This being said, in general, not all chains of  ${N}(\cP)$ are well ordered. However, when ${N}(\cP)$ has finite length, namely all of its chains are finite,
then all chains of ${N}(\cP)$ are well ordered.

\begin{lemma}
	\label{ln}
	Let $\cP$ be a non-degenerate embeddable polar space of
	rank at least $2$  with relatively universal embedding $\varepsilon$ and let $\cX$
	be a nice subspace of $\cP$. Put $[X]=\langle\varepsilon(\cX)\rangle$.
	\begin{enumerate}
		\item
		      If $\cX$ is not generated by a frame,
		      then there
		      is $Y\subset X$ such that $\varepsilon^{-1}(Y)$ is nice, $\dim(X/Y)=1$ and
		      $\langle\varepsilon(\varepsilon^{-1}([Y]))\rangle=[Y]$.
		\item
		      If $\cX\subset\cP$ is a proper nice subspace of $\cP$,
		      then there is $Z$ such that $X\subset Z$,
                      $\dim(Z/X)=1$ and $\langle\varepsilon(\varepsilon^{-1}([Z]))\rangle=[Z]$.
	\end{enumerate}
\end{lemma}
\begin{proof}
	Let $(p_1,p_1',\dots,p_n,p_n')$ be a frame $\cF$ of $\cP$ contained in $\cX$. Then $X$ admits a basis
$B:=(\varepsilon(p_1),\varepsilon(p_1'),\dots,\varepsilon(p_n),\varepsilon(p_n'),\varepsilon(e_1),\dots,\varepsilon(e_{\lambda-1}),
		\varepsilon(e_{\lambda}))$.
	\begin{enumerate}
		\item
		      Put $Y:=\langle
			      \varepsilon(p_1),\varepsilon(p_1'),\dots,\varepsilon(p_n),\varepsilon(p_n'),\varepsilon(e_1),\dots,\varepsilon(e_{\lambda-1})\rangle$.
		      By construction $\dim(X/Y)=1$ and $\cY:=\varepsilon^{-1}(Y)$ is nice, as it contains the frame $\cF$.
		      Clearly, $\cY\neq\cX$ (as $e_{\lambda}\in\cX\setminus\cY$), so $\cY$ is a proper hyperplane of $\cX$ and, by construction,
		      $\langle\varepsilon(\cY)\rangle=[Y]$.
		\item
		      Suppose now that $\cX$ is a proper nice subspace of $\cP$ and $V$ is  the space hosting $\eps(\cP)$. Since  $\cX$ arises from the embedding $\eps$,
		      we have $X\neq V$; so there is at least one point $z\in\cP\setminus\cX$ such that $\varepsilon(z)\not\in X$. Put $Z=\langle X,\varepsilon(z)\rangle$.
	\end{enumerate}
	The thesis follows in both cases.
\end{proof}

The following lemma will be used in Section~\ref{parabolic gap}.

\begin{lemma}
	\label{parabolic defect}
	Let $\cP$ be a non-degenerate embeddable polar space of rank at least $2$ and let $\cX$
	be a nice subspace of $\cP$. Then there is a maximal chain
	of nice subspaces of $\cP$ having $\cX$ as an element.
\end{lemma}
\begin{proof}
	We may assume that $\cP$ is neither a grid nor a bi-embeddable quaternion
	quadrangle, otherwise $\cP$ is its unique nice subspace and there is nothing
	to prove.
	By Theorem~\ref{thm subspaces}, all nice subspaces of $\cP$
	arise from the universal embedding of $\cP$. Let $\cF$ be a frame
	of $\cX$ and put $[F]:=\langle\varepsilon(\cF)\rangle$.
	Let also $[X]:=\langle\varepsilon(\cX)\rangle$.
	Let $B_0=(\mathbf{f}_1,\dots,\mathbf{f}_{2n})$ be a basis of $F$,
	where $[\mathbf{f}_i]=\varepsilon(f_i)$ with $f_i\in\cF$.
	As $[X]$ is spanned by $\varepsilon(\cX)$, it is possible to
	complete $B_0$ to a basis of $X$ by adding vectors from
	$B_1:={(\mathbf{x}_i)}_{i<\delta}$,
	where $\delta$ is a suitable ordinal number and
	$[\mathbf{x}_i]=\varepsilon(x_i)$ with $x_i\in\cX$.
	Finally, the basis $B_0\cup B_1$ can be completed to a basis of
	$V$ by adding further vectors from
	$B_2={(\mathbf{v}_j)}_{j<\psi}$
	with $[\mathbf{v}_i]=\varepsilon(v_i)$ with $v_i\in\cP$ and $\psi$ a suitable ordinal number.
	The sets $B_0,B_1$ and $B_2$ admit a well order by
	the axiom of choice. Consequently, the set
	$B=B_0\cup B_1\cup B_2$ is also well ordered if we put
	each element of $B_0$ before each element of
	$B_1$ and each element of $B_1$ before each element of $B_2$.
	For any $\gamma<\delta$ and $\xi<\psi$,  let
	\[ \mathfrak V: F\subset L_1\subset L_2\dots \subset X\subset
		U_1\subset U_2\subset\dots\subset V \]
	be the chain of subspaces of $V$ given by
	\[ L_{\gamma}:=\langle F\cup {\{\mathbf{x}_{i}\}}_{i<\gamma}\rangle
          \qquad\text{and}
		\qquad
		U_{\xi}:=\langle X\cup{\{\mathbf{v}_{j}\}}_{j<\xi}\rangle.
	\]
	We have
	$L_0=F$, $L_{\delta}=X=U_0$ and $U_{\psi}=V$.
	
	By Lemma~\ref{ln},
	we can always take $\mathfrak V$ to be  a maximal chain, since
	for any two successive spaces, say $V_{i+1}$ and $V_i$, we
	have $\dim(V_{i+1}/V_i)=1$.
	Also, the chain $\mathfrak V$ contains $X$.
	Put now $\cL_i:=\varepsilon^{-1}([L_i])$
	and  $\cU_j:=\varepsilon^{-1}([U_j])$.
	We then obtain a well-ordered chain of subspaces:
	\[ \cF\subset\cL_1\subset\dots\subset\cL_{\delta}=\cX=\cU_0\subset\dots
		\subset\cU_{\psi}. \]
	Observe that any two terms of this chain are (by construction) different
	and each of them is a hyperplane in the term which follows
	(since all subspaces arise from $\varepsilon$
        by Theorem~\ref{thm subspaces}).
	So, this chain is maximal
	and it contains $\cX$. In particular, its length is $
        |\delta|+|\psi|$ according to the definition of length of a chain
        given in Section 1.
\end{proof}

\section{Hyperbolic lines}\label{hyperb lines}
Suppose that $\cP$ is a non-degenerate embeddable polar space of rank $n\geq2$.
For any pair of non-collinear points $a,b\in\cP$, the set ${\{a,b\}}^{\perp\perp}$ is the \emph{hyperbolic line} determined by $a$ and $b$. Clearly, $a,b\in{\{a,b\}}^{\perp\perp}$
for any $a,b\in\cP$ and any two points of ${\{a,b\}}^{\perp\perp}$ are non-collinear.  In particular a hyperbolic line of $\cP$ always
contains at least $2$ points.

Let now $\varepsilon:\cP\to\PG(V)$ be the universal embedding of $\cP$
if $\cP$ is not of type (E\ref{E1}) or
(E\ref{E2}) or one of its relatively universal embeddings in the latter cases.
If $\varepsilon(\cP)$ is described by a hermitian or quadratic
form $ \phi\colon V\rightarrow \KK$, so that $\varepsilon(\cP)=:\cP(\phi)$,
then let $f:V\times V\rightarrow \KK$ be the sesquilinearization of $\phi$; when $\varepsilon (\cP)$ is described by an alternating
form, let $f$ be that form;
in either case $\perp_{f}$ is  the polarity
induced  by $f$ in $\PG(V)$.

For $a,b\in\cP$ put $[\mathbf{a}]=\varepsilon(a)$ and
$[\mathbf{b}]=\varepsilon(b)$ with $\mathbf{a},\mathbf{b}\in V$.
We have $a\perp b$ if and only if $[\mathbf{a}]\perp_f[\mathbf{b}]$ if and only if $f(\mathbf{a}, \mathbf{b})=0$. 
More explicitly, any two given points
$[\mathbf{a}],[\mathbf{b}]\in\cP(\phi)$ with
$[\mathbf{a}]\not\perp[\mathbf{b}]$ determine the hyperbolic line
of $\cP(\phi)$ consisting of
${\{[\mathbf{a}],[\mathbf{b}]\}}^{\perp_f\perp_f}\cap \cP(\phi)$.
So, any hyperbolic line ${\{a,b\}}^{\perp\perp}$ of $\cP$ is in correspondence with the  unique hyperbolic line ${\{\varepsilon(a),\varepsilon(b)\}}^{\perp_f\perp_f}\cap \cP(\phi)$ of $\cP(\phi)$ through $\varepsilon$:
\begin{equation}\label{hyperbolic lines-1}
{\{a,b\}}^{\perp\perp}=\varepsilon^{-1}({\{\varepsilon(a),\varepsilon(b)\}}^{\perp_f\perp_f}\cap\varepsilon(\cP)).
\end{equation}
With the notation introduced before, we have the following.
\begin{lemma}\label{lemma1}
	The embedding $\varepsilon$ induces a bijection between the set of
	hyperbolic lines of $\cP$ and those of $\cP(\phi)$.
\end{lemma}
If $\eps(\cP)$ is a symplectic polar space, then, necessarily,
$\ch(\KK)\neq 2$ and any hyperbolic line of $\eps(\cP)$ is a line of $\PG(V)$.
If $\eps(\cP)=\cP(\phi)$ with $\phi$ hermitian with a non-degenerate
sesquilinearization $f$,
then any of its hyperbolic lines
is a subline; explicitly if $a$ and $b$ are non-collinear points of $\cP$
and the representative vectors $\mathbf{a}$ and $\mathbf{b}$ of
$\varepsilon(a)$ and $\varepsilon(b)$ are chosen in such a way that
$f(\mathbf{a},\mathbf{b})=1$, then ${\{\varepsilon(a),\varepsilon(b)\}}^{\perp_f\perp_f}$ is the set
$\{[\mathbf{b}]\}\cup\{ [\mathbf{a}+\mathbf{b}t] \colon t^{\sigma}+t=0\}$,
where $\sigma$ is the anti-automorphism of $\KK$ associated to $\phi$.

When $\varepsilon(\cP)=\cP(\phi)$ for a quadratic form
$\phi$ with a non-degenerate bilinearization, then every hyperbolic line of
$\varepsilon(\cP)$ consists of just $2$ points.

When $\varepsilon(\cP)=\cP(\phi)$ is hermitian or quadratic with a sesquilinearization $f$ and $R=\mathrm{Rad}(f)\neq\{\mathbf{0}\}$, then
every hyperbolic line ${\{\varepsilon(a),\varepsilon(b)\}}^{\perp_f\perp_f}$
of $\varepsilon(\cP)$ spans a $(\dim(R)+2)$--dimensional subspace
$\langle\varepsilon(a),\varepsilon(b),R\rangle$ of $V$.

\begin{lemma}
	\label{lemb}
	With $\cP$ and $\varepsilon$ as above,
	if all the hyperbolic lines of $\cP$ contain exactly two points, then
	$\varepsilon (\cP)$
	is a non-degenerate
	orthogonal polar space of either hyperbolic or elliptic type.
\end{lemma}

\begin{proof}
	According to the notation introduced at the beginning of this section, let $a$ and $b$ be two non-collinear points of $\cP$.
	By hypothesis, we have that
	$\{a,b\}^{\perp\perp}=\{a,b\}$.
	By Equation~\eqref{hyperbolic lines-1}, it follows that
	$\{[\mathbf{a}],[\mathbf{b}]\}^{\perp_f\perp_f}\cap\cP(\phi)=
		\{[\mathbf{a}],[\mathbf{b}]\}$. So, by Lemma~\ref{lemma1},
                we have that all hyperbolic lines of $\cP(\phi)$ contain exactly two points.

	\medskip
	{\sc Case A}: $f$ is non-degenerate. Then
	for
	any two distinct non-orthogonal points $[\mathbf{a}],[\mathbf{b}]\in \eps(\cP)$,
	we have that $\{\mathbf{a},\mathbf{b}\}^{\perp_f\perp_f}$ spans the projective line through $[\mathbf{a}]$ and $[\mathbf{b}]$, i.e.\
	$\dim\{\mathbf{a},\mathbf{b}\}^{\perp_f\perp_f}=2$.
	So, a projective line is a $2$-secant for $\eps(\cP)$
	if and only if it is spanned by a hyperbolic line of $\eps(\cP)$
	(equivalently, it is spanned by the image of a hyperbolic line of $\cP$).

        Let $\ell$ be a line of $\PG(V)$ containing at least three distinct points $[\mathbf{u}],[\mathbf{v}],[\mathbf{w}]$ of $\eps(\cP)$
	and suppose by way of contradiction that $\ell$ is not contained in $\eps(\cP)$.
	Then $\mathbf{u},\mathbf{v},\mathbf{w}$ are pairwise non-orthogonal
	vectors and
	the hyperbolic line determined by $\mathbf{u}$ and $\mathbf{v}$
	coincides with the hyperbolic line determined by
	$\mathbf{u}$ and $\mathbf{w}$, i.e.\ $\{[\mathbf{u}],[\mathbf{v}]\}^{\perp_f \perp_f}=\{[\mathbf{u}],[\mathbf{w}]\}^{\perp_f \perp_f}$. This is a contradiction because the hyperbolic line
$\{u,v\}^{\perp\perp}=\varepsilon^{-1}(\{[\mathbf{u}],[\mathbf{v}]\}^{\perp_f \perp_f})$ would contain at least three points, namely $u,v,w$.
	So, $\ell\subseteq \eps(\cP).$
	In other words,
	any projective line intersecting $\eps(\cP)$ in at least three points must be contained in $\eps(\cP)$. This property together with the fact that $\eps(\cP)$
	spans $\PG(V)$ proves that $\eps(\cP)$ fulfills the definition of
	a {\it Tallini set}.   By~\cite[Theorem 3.8]{d02}
        (see also~\cite[Theorem 1.1]{B69}), we have that $\varepsilon(\cP)$ is described by a non-degenerate quadratic form, i.e.\ $\cP$ is an orthogonal polar space.
	Moreover, since we are assuming that $f$ is non-degenerate, $\eps(\cP)$ is an orthogonal polar space of either hyperbolic or elliptic type, according
	to our definitions.
	\medskip
	
	{\sc Case B}: $f$ is degenerate.
	In this case, we have that $[X]:=\{[\mathbf{u}],[\mathbf{v}]\}^{\perp_f\perp_f}=[\langle \mathbf{u},\mathbf{v}\rangle+\mathrm{Rad}(f)]$, where $[\mathbf{u}]\not\perp_f
		[\mathbf{v}]$ and $\dim(\mathrm{Rad}(f))\geq 1$.
	Put $R:=\mathrm{Rad}(f)$.
	Since $R\cap\langle\mathbf{u},\mathbf{v}\rangle=\{\mathbf{0}\}$,
	we have
	$\dim(\{\mathbf{u},\mathbf{v}\}^{\perp_f\perp_f})>2$, i.e.\ $\{[\mathbf{u}],[\mathbf{v}]\}^{\perp_f\perp_f}$
	contains at least a projective plane of $\PG(V)$.
	
	Let
	$\phi_X$ be the form induced by $\phi$ on $X$.
	Since $R\subseteq X$, we have $\langle \mathbf{u},\mathbf{v}\rangle\subset X$.
	Also, $\mathbf{u},\mathbf{v}\not\in R$ (as $\mathbf{u}\not\perp_f\mathbf{v}$)
	and $\phi_X(\mathbf{u})=0=\phi_X(\mathbf{v})$. By~\cite[8.2.7]{T},
        $X$ is generated by its $\phi_X$-singular vectors. On the other
	hand, the singular points of $[X]$ are those of $[X]\cap\cP(\phi)$
	and $|[X]\cap\cP(\phi)|=2$. This is a contradiction.
	This forces $f$ not to be degenerate.

        By Case A, the lemma is proved.
\end{proof}

\begin{lemma}
	\label{l1}
	Let $\cP$ and $\varepsilon:\cP\to\PG(V)$  be as above.
	If $\eps(\cP)$ is  orthogonal of either hyperbolic or elliptic
        type,
	then every hyperbolic line of $\cP$ contains exactly $2$ points.
\end{lemma}
\begin{proof}
	According to the notation introduced at the beginning of this section, since $\eps(\cP)$ is
	either hyperbolic or elliptic, we have
	$\mathrm{Rad}(f)=\{\mathbf{0}\}$.
	Take two arbitrary distinct points $a,b$ of $\cP$ which are not collinear in $\cP$ and let $\varepsilon(a)=[\mathbf{a}]$, $\varepsilon(b)=[\mathbf{b}]$ with
	$\mathbf{a},\mathbf{b}\in V$.
	Then
	$\phi(\mathbf{a})=\phi(\mathbf{b})=0$ and $f(\mathbf{a},\mathbf{b})\neq0$.
	Since $f$ is non-degenerate,
	$\dim \{\mathbf{a}, \mathbf{b}\}^{\perp_f\perp_f}=2$ and this space
	contains $\{\mathbf{a},\mathbf{b}\}$.
	Since $\eps(\cP)$ is a hypersurface of degree $2$, either the line $[
				\langle\mathbf{a}, \mathbf{b}\rangle]$ is a $2$-secant to $\cP(\phi)$ or it is contained in $\cP(\phi)$.
	If
	$\{[\mathbf{a}],[\mathbf{b}]\}^{\perp_f\perp_f}\subset\cP(\phi)$, then
	$f(\mathbf{a},\mathbf{b})=0$
	which is a contradiction, for this would imply $a\perp b$ in
	$\cP$.
	Hence $\{[\mathbf{a}],[\mathbf{b}]\}^{\perp_f\perp_f}\cap\cP(\phi)=
		\{[\mathbf{a}],[\mathbf{b}]\}$.
	In particular, $\{a,b\}^{\perp\perp}=\varepsilon^{-1}(\{[\mathbf{a}],[\mathbf{b}]\}^{\perp_f\perp_f}\cap\cP(\phi))=\{a,b\}$
	and all hyperbolic lines of $\cP$ consist of $2$ points.
\end{proof}

\begin{proof}[Proof of Theorem~\ref{thm hyperb lines}]
	The theorem follows directly from Lemma~\ref{lemb} and Lemma~\ref{l1}.
\end{proof}
\begin{proof}[Proof of Corollary~\ref{pr:orth}]
	By Lemma~\ref{lemb}, there exists a quadratic form $\bar{\phi}\colon \langle \varepsilon(\cF)\rangle\rightarrow \KK$ such that $\eps|_{\cF}(\cF)=\cF(\bar{\phi})$, where $\eps$ is  the universal embedding of $\cP$.
	As $\eps|_{\cF}$ is the restriction to $\cF$ of $\eps$, it follows that
	$\eps(\cP)$ must be described by
	a suitable quadratic form $\phi$ whose restriction to
        $\langle \varepsilon(\cF)\rangle$
	is exactly $\bar{\phi}$.
	
	Conversely, suppose that
	$\varepsilon(\cP)=\cP(\phi)$ with $\phi$ a non-degenerate
	quadratic form. Let $F$ be a frame of $\cP$ and $\cF$ be
	the subspace of $\cP$ spanned by $F$. Then
	$\varepsilon(\cF)=[\langle\varepsilon(F)\rangle]\cap\varepsilon(\cP)$
	and the latter is
	a hyperbolic quadric in $[\langle\varepsilon(F)\rangle]$.
	In particular, all hyperbolic lines of $\varepsilon(\cF)$
	consist of exactly $2$ points. By Lemma~\ref{lemma1}, we now have
	that $\cF$ is hyperbolic and so all of its hyperbolic lines consist of
	$2$ points.
\end{proof}

\section{Elliptic gap}\label{sec 3}
Recall from  Definition~\ref{dec} that an elliptic chain of a non-degenerate polar space $\cP$ of finite rank is
a well-ordered chain
\[ \fE: \cF=\cE_0\subset\cE_1\subset\dots\subset\cE_{\delta}\subset\dots,\]
where $\cE_{i}\subseteq\cP$ is a nice subspace of $\cP$ with the property that any of its hyperbolic lines contains exactly $2$ points.

The following lemma deals with arbitrary non-degenerate polar spaces.
\begin{lemma}
	\label{l:hs}
	Let $\cP$ be a non-degenerate polar space and $(\cS_i)_{i\in I}$
	be a chain of non-degenerate
	polar subspaces of $\cP$ ordered with respect to the inclusion relation $\subseteq$.
	Then
	$\cS_I:=\bigcup_{i\in I}\cS_i$ is a non-degenerate polar space
	and $\cS_i$ is a subspace of $\cS_I$ for all $i$.
	For any $a,b\in\cS_I$ with $a\not\perp_I b$ we have
	\[ \{a,b\}^{\perp_I\perp_I}\subseteq
		\bigcup_{i:a,b\in\cS_i}\{a,b\}^{\perp_i\perp_i}, \]
	where $\perp_x$ denotes the collinearity relation in $\cS_x$
	for $x=I$ or $x\in I$.
\end{lemma}
\begin{proof}
	By definition of subspace,
	$a\perp_I b$ in $\cS_I$ if and only if for all
	$i\in I$ such that $a,b\in\cS_i$ we have $a\perp_i b$.
	Let $a,b\in\cS_I$ with $a\not\perp_Ib$.
	Take $c\in\{a\}^{\perp_I\perp_I}$.
	For any $j$ such that $a,b,c\in\cS_j$ and any $x\in\{a,b\}^{\perp_j}$
	we get $x\in\{a,b\}^{\perp_I}$; so,
	$c\perp_I x$ and, consequently, $c\perp_i x$;
	this implies that $c\in\{a,b\}^{\perp_j\perp_j}$.
\end{proof}
The following corollary is a direct consequence of the previous lemma.
\begin{corollary}
	\label{e:max}
	Let $\cP$ be a non-degenerate orthogonal polar space.
	\begin{enumerate}
		\item\label{max1}
		If $\fE$ be a chain of elliptic subspaces of $\cP$, then $\cE':=\bigcup_{\cE}\fE$ is
		either an elliptic space or it is a subspace generated by a frame.
		\item If $\cE$ is an elliptic polar subspace of $\cP$, then $\cE$ is contained in a maximal
		      elliptic polar subspace $\cE'$ of $\cP$.
	\end{enumerate}
\end{corollary}
\begin{proof}
	\begin{enumerate}
		\item\label{ec1}
		By Lemma~\ref{l:hs}, the hyperbolic lines of $\cE'$ consists of just two
		points. Then the first statement  follows from Theorem~\ref{thm hyperb lines}.
		\item
		      Consider the set of all elliptic polar subspaces of $\cP$
		      containing $\cE$. By the first statement  every chain in it has an upper bound. Then the result now follows from Zorn's lemma.
	\end{enumerate}
\end{proof}

\begin{lemma}
	Let $\cP$ be a non-degenerate orthogonal
	polar space either defined over a field
	$\KK$ with $\ch(\KK)\neq 2$ or a hyperbolic orthogonal polar space.
	Then there are  maximal elliptic chains of $\cP$
	and these are exactly the maximal well ordered
        chains of nice subspaces of $\cP$.
\end{lemma}
\begin{proof}
	If $\cP$ is hyperbolic, then it is generated by a frame and
	there is nothing to prove.
	Suppose now $\cP$ orthogonal and $\ch(\KK)\neq2$; then
	all nice subspaces $\cE$ of $\cP$ have hyperbolic lines of size
	$2$ so, they are \emph{elliptic} according to our definition.
	It follows that
	maximal elliptic chains of $\cP$ and
	maximal well ordered chains of nice subspaces of $\cP$ are
	the same thing.
	So the lemma follows.
\end{proof}

Note that, when $\ch(\KK)=2$, a \emph{proper} (i.e.\ containing more than
one term) elliptic chain is not a
maximal well ordered chain  of nice subspaces of $\cP$.

\begin{setting}\label{notazione}
	\rm{Henceforth, throughout this section we fix the following notation:
		$\cP$ is a non-degenerate
		embeddable polar space defined over a
                field of characteristic $2$
		and $\varepsilon\colon \cP \rightarrow \PG(V)$,
                where $V$ is defined over a field $\KK$, is
		either its universal embedding if $\cP$ is not a grid or
		any of its relatively universal embeddings if $\cP$ is a grid
		(case~(E\ref{E1})).
		The image of $\varepsilon$ is $\varepsilon(\cP)=\cP(\phi)$,
		where
		$\phi\colon V\rightarrow \KK$ is a quadratic form
		having bilinearization $f\colon V\times V\rightarrow \KK$
		(see Section~\ref{sesquilinear and pseudoquadratic form}).}
	Observe that the bi-embeddable quaternion quadrangles mentioned in case (E\ref{E2}) are excluded because $\KK$ is a field.
\end{setting}

Suppose that $\cS$ is a proper subspace of $\cP$. Put $[W]:=\langle\varepsilon(\cS)\rangle$.
By Theorem~\ref{thm subspaces}, $\cS=\varepsilon^{-1}([W])$ and
each subspace $\cS$ of non-degenerate rank at least $2$ arises
from a subspace $[W]$ of $\PG(V)$; see~\cite{ILP21a}.
Since $\cP$ is non-degenerate
of rank $n\geq 2$, all subspaces of $\cP$ containing a frame
have also non-degenerate rank $n\geq2$ so they arise from subspaces of
$\PG(V)$.
Let $R_W=\mathrm{Rad}(f_W)$ be the radical of the bilinearization
$f_W\colon W\times W\rightarrow \KK$ of the quadratic form $\phi_W\colon W\rightarrow \KK$ induced by $\phi$ on $W$.

\begin{lemma}
	\label{l4}
	Let $\cP$ be elliptic and $\cS$ be a
	maximal proper nice subspace of $\cP$.
	Then $\dim(R_W)=1$.
\end{lemma}
\begin{proof}
	Suppose $F=\{F_1,F_2\}$ is a frame of $\cP$ contained in $\cS$. Hence $F_1$ and $F_2$ are maximal singular subspaces of $\cP$ and $\varepsilon(F_1)$ and $\varepsilon (F_2)$ are maximal singular subspaces of $\varepsilon(\cP)$.
	
	Since $\cS$ is a maximal subspace of $\cP$, it is indeed a hyperplane of $\cP$ and since it contains a frame, by Remark~\ref{nice hyperplane}, we have that $\cS$ is a non-singular hyperplane of $\cP$. Hence $[W]$ is a hyperplane of $\PG(V)$.
	Since $\cP$ is elliptic, $\mathrm{Rad}(f)=\{\mathbf{0}\}$ and
	the polarity induced by $f$ is non-degenerate; so, for
	every hyperplane $\Sigma$ of $\PG(V)$ there is a point $[\mathbf{p}]$ of $\PG(V)$ such that $\Sigma=[\mathbf{p}]^{\perp_f}$.
	So, $[W]=[\mathbf{p}]^{\perp_f}$ for some point $[\mathbf{p}]\in \PG(V)$.
	The point $[\mathbf{p}]$ cannot be in $\varepsilon(\cP)$, otherwise the
	frames of $\varepsilon(\cP)$ contained in $\varepsilon(\cE)$,
	which exist since $\cS$ is nice by assumption, would be
	contained in  $[\mathbf{p}]^{\perp}$ but no singular hyperplane
	contains a frame, by Remark~\ref{nice hyperplane}.
	Thus, either $[\mathbf{p}]\in\PG(V)\setminus([W]\cup\varepsilon(\cP))$ or
	$[\mathbf{p}]\in [W]\setminus\varepsilon(\cP)$.
	
	Suppose the former.
	Then
	$[\mathbf{p}]$ (note that $[\mathbf{p}]\not\in \varepsilon (\cS)$) is orthogonal to every point of $[W]=[\mathbf{p}]^{\perp_f}$, hence $V=W\oplus \langle\mathbf{p}\rangle$.
	Take $s\in \cP$. Then $\varepsilon (s)=[\mathbf{w}+\alpha\mathbf{p}]$,
	where $\mathbf{w}\in W$ and $\alpha\in \KK$. We have
	\[f(\mathbf{p},\mathbf{w}+\alpha \mathbf{p})=
		f(\mathbf{p},\mathbf{w})+\alpha f(\mathbf{p},\mathbf{p})=0,\]
	being
	$f(\mathbf{p},\mathbf{p})=0$ (recall $\ch(\FF)=2$)  and $f(\mathbf{p},\mathbf{w})=0$ because
	$[\mathbf{p}]\perp_f[\mathbf{w}]$.
	So, $\varepsilon (s)\in [\mathbf{p}]^{\perp_f}$ hence $\varepsilon (\cP)\subseteq[W]$. This is clearly not possible since $\varepsilon (\cP)$ spans $\PG(V)$ and the equality
	$\cS=\varepsilon^{-1}([W])$ forces
	$\cP\subseteq\cS$, a contradiction.
	
	So, only the latter case remains to consider. Hence, $[\mathbf{p}]\in [W]=[\mathbf{p}]^{\perp_f}$. This means that the bilinear form $f_W\colon W\times W\rightarrow \FF$ induced by $f$ on $W$ is degenerate, implying $\dim(R_W)\geq 1$.
	
	Suppose $\dim(R_W)\geq 2$.
        Then there exist $\mathbf{r_1},\mathbf{r_2}\in R_W$ such
	that $\varepsilon(\cS)\subseteq
        [\mathbf{r_1}]^{\perp_f}\cap[\mathbf{r_2}]^{\perp_f}$.
	This means that the codimension  of $W$ in $V$ is at least $2$,
	since $f$ is non-degenerate on $V$.
	So, $[W]$ is not a hyperplane which is a contradiction.
        Thus, $\dim(R_W)=1$.
\end{proof}

\begin{lemma}
	\label{l5}
	Let $\cP$ be  orthogonal.
	If $\dim(\mathrm{Rad}(f))=1$, then there exists a
	nice elliptic hyperplane $\cS$ of $\cP$.
\end{lemma}
\begin{proof}
	Since $\langle \varepsilon (\cP)\rangle=\PG(V)$,
	there is a basis $B$ of $V$ consisting of elements
	\[ B=(\mathbf{e_1},\mathbf{e_2},\dots,\mathbf{e_{\delta}},\dots) \]
	with $e_1,\dots,e_{\delta },\dots\in\cP$ and $\varepsilon(e_i)=[\mathbf{e_i}]$
	and such that
	the first $2n$ vectors of $B$ determine a frame of $\cP$.
	Let $\mathbf{r}$ be a vector such that $\mathrm{Rad}(f)=\langle\mathbf{r}\rangle$.
	Let $\phi'$ be the quadratic form induced by $\phi$ on
	$X:=\langle\mathbf{e}_1,\dots,\mathbf{e}_{2n}\rangle$ and $\cF(\phi')$
	be the polar space defined by $\phi'$ on $[X]$.
	Then $\cF(\phi')=\varepsilon(\varepsilon^{-1}([X]))$ is a hyperbolic
	quadric and consequently  the bilinearization $f'$ of $\phi'$ (which is
	the restriction to $X\times X$ of the bilinearization $f$ of $\phi$)
	is non-degenerate. Therefore, $\mathbf{r}\not\in X$.
	
	Since $B$ is a basis of $V$,
	$\mathbf{r}$ can be written in an unique way as a linear
	combination of a finite number  of elements (at least $2$) of $B$
	with nonzero coefficients, at least one of which has index $i>2n$ in
	$B$; up to reordering the elements of $B$
	we can assume without loss of generality that
	$\mathbf{r}=\mathbf{e_{2n+1}}+\sum_{i\neq 2n+1}\alpha_i\mathbf{e_i}$, where $\alpha_i\neq 0$ for at most
	a finite number of indexes $i$.
	Put then $\overline{V}=\langle B\setminus\{\mathbf{e_{2n+1}}\}\rangle$.
	Clearly, we have that
	\[ V=\overline{V}\oplus\mathrm{Rad}(f), \]
	the space $[\overline{V}]$ is a hyperplane in $\PG(V)$ and the hyperplane of
	$\cP$ given by
	$\cS:=\varepsilon^{-1}([\overline{V}])$ contains a frame and is such that
	$\langle\varepsilon(\cS)\rangle=[\overline{V}]$.
	So, $\cS$ is a nice hyperplane of $\cP$ and, by
	Remark~\ref{nice hyperplane}, it is non-singular.
	We claim that $\cS$ is elliptic.
        By the above arguments,
        we already know that
	$\varepsilon(\cS)=\cS(\phi|_{\overline{V}})$ and the quadratic form
	$\phi|_{\overline{V}}$ is non-degenerate. It is
	straightforward to see that the bilinearization of $\phi|_{\overline{V}}$ is $f|_{\overline{V}\times\overline{V}}$.
	By way of contradiction suppose now that $\mathrm{Rad}(f|_{\overline{V}\times\overline{V}})$
	is not trivial.
	Since $\varepsilon(\cS)\subseteq[\overline{V}]$, there exists a non-null vector $\mathbf{w}\in\mathrm{Rad}(f|_{\overline{V}\times\overline{V}})\subseteq \overline{V}$ such that
	$\varepsilon(\cS)\subseteq [\mathbf{w}]^{\perp_f}$. By $\mathbf{w}\in \overline{V}$ and $\overline{V}\cap\mathrm{Rad}(f)=\{\mathbf{0}\}$,
	$\mathbf{w}\not\in\mathrm{Rad}(f)$, it follows that $[\mathbf{w}]^{\perp_f}$ is a hyperplane of $[V]$.
	Also, $[\mathrm{Rad}(f)]\in [\mathbf{w}]^{\perp_f}$  gives
	$[\mathbf{w}]^{\perp_f}\neq[\overline{V}]$.
	On the other hand, $\varepsilon(\cS)\subseteq[\overline{V}]\cap [\mathbf{w}]^{\perp_f}$, implying that $\langle\varepsilon(\cS)\rangle$ is not  a hyperplane of $\PG(V)$. Contradiction.
	Hence $\mathrm{Rad}(f|_{\overline{V}\times\overline{V}})=\{\mathbf{0}\}$
	and $\cS$ is  elliptic.	
\end{proof}

\begin{corollary}
	\label{c1}
	Let $\cP$ be elliptic and $\cS$ be a maximal proper nice subspace.
	Then there exists  a nice elliptic hyperplane of $\cS$.
\end{corollary}
\begin{proof}
	By Lemma~\ref{l4}, $\dim(\mathrm{Rad}(f|_{\langle \varepsilon(\cS)\rangle\times
			\langle\varepsilon(\cS)\rangle}))=1$.
                      Now, applying Lemma~\ref{l5} to $\cS$ we have that
                      there exists
                      a nice elliptic hyperplane of $\cS$.
\end{proof}

Corollary~\ref{c1} shows
that if $\cE_{i}$ and $\cE_{i+1}$ are both elliptic polar spaces
(or for $i=0$, $\cE_i$ is the subspace generated by a frame)
with
$\cE_{i}\subset\cE_{i+1}$ and there are no elliptic subspaces
between $\cE_{i}$ and $\cE_{i+1}$,
then there exists a nice polar space $\cP_i$ sitting between $\cE_{i}$ and $\cE_{i+1}$: $\cE_{i}\subset \cP_i \subset\cE_{i+1}$ such that $\cP_i$ is a hyperplane of $\cE_{i+1}$ and
$\cE_{i}$ is a hyperplane of $\cP_i$.
Consequently, $\dim(\langle\varepsilon(\cE_{i+1})\rangle/\langle\varepsilon(\cE_{i})\rangle)
\geq2$.

\begin{theorem}
	\label{t:exists}
	Any non-degenerate orthogonal polar space $\cP$ admits
	well ordered maximal elliptic chains $\fE$ of subspaces.
\end{theorem}
\begin{proof}
	Let $\mathbf{C}$ be the set of all well-ordered chains
	$\fE=(\cE_{\gamma})_{\gamma<\omega}$ of nice subspaces of $\cP$ such that the
	following hold:
	\begin{enumerate}[(C1)]
		\item\label{cc1} The first element of $\fE$ is generated by a frame and all the other elements $\cE_{\gamma}$ are elliptic  subspaces of $\cP$.
		\item\label{cc2} If $\cE_{\gamma}$ and
		$\cE_{\gamma+1}$ are consecutive elements
		of $\fE$, then $\langle\varepsilon(\cE_{\gamma})\rangle$ has
		codimension $2$ in $\langle\varepsilon(\cE_{\gamma+1})\rangle$.
		\item\label{cc3} If $\gamma<\omega$ is a limit ordinal, then
		$\bigcup_{\xi<\gamma}\cE_{\xi}=\cE_{\gamma}$.
	\end{enumerate}
	By construction, the set $\mathbf{C}$ is non-empty.
	Suppose $\fE,\fE'\in\mathbf{C}$ with $\fE\subset\fE'$.
	Let $\delta$ be such that $\cE'_{\gamma}\in\fE$ for all
	$\gamma<\delta$ and $\cE'_{\delta}\in\fE'\setminus\fE$.
	Using
	Lemmas~\ref{l4} and~\ref{l5} and conditions (C\ref{cc1}), (C\ref{cc2}) and
	(C\ref{cc3})
	we see that $\cE_{\mu}'=\cE_{\mu}$ for all $\mu<\delta$; so
	$\fE$ is an initial segment for $\fE'$, i.e.\ $\fE\subseteq\fE'$ and
	for all $\cE'\in\fE'\setminus\fE$ and for all $\cE\in\fE$ we
	have $\cE\subseteq\cE'$.
	
	We need to show that
	the set $\mathbf{C}$ contains maximal elements with respect to
	inclusion.
	
	Let $(\fE_i)_{i\in I}$ be a chain of elements
	of $\mathbf{C}$ and put $\fE_I:=\bigcup_{i\in I}\fE_i$. We first
	claim that $\fE_I$ is well ordered. Indeed, let $X\subseteq\fE_I$
	be non-empty. For any $\cE_0\in X$ take $i\in I$ such
	that $\cE_0\in\fE_i$; clearly, the minimum of $X$ cannot properly contain
	$\cE_0$, so it is enough to show that
	${X}_0:=\{\cE\in X:
		\cE\subseteq\cE_0\}$ admits minimum. On the other hand,
	${X}_0\subseteq\fE_i$ and $\fE_i$ is well ordered,
	so the minimum exist.
	
	We now prove that $\fE_I\in\mathbf{C}$. Indeed,
	$\fE_I$ satisfies
	(C\ref{cc1}) by construction.
	On the other hand, if $\cE\in\fE_i$ and $\cE'\in\fE_j$ with $i,j\in I$ and
	$\cE'\subseteq\cE$, then $\cE'\in\fE_i$. In particular,
	$\fE_i$ is an initial segment to $\fE_I$ for all $i$. Since
	(C\ref{cc2}) and (C\ref{cc3}) hold in any $\fE_i$, we get
	that they must also hold in $\fE_I$.
	Thus, we have that $\mathbf{C}$ contains maximal chains by Zorn's Lemma.
	
	We now prove that a maximal chain $\fE$ in $\mathbf{C}$ is also
	a maximal elliptic chain.
	First observe that $\cE_{\fE}:=\bigcup_{\cE\in\fE}\cE$ is an
	elliptic subspace; also $\cE_{\fE}\in\fE$, since $\fE$ is
	maximal. So, any maximal chain in $\mathbf{C}$ has a
	maximum.
	Let us suppose that $\fE$ is not a maximal elliptic chain;
	then there is an elliptic space $\mathcal T$ which can be
	added to $\fE$ and such that $\cE_{\fE}\subset{\mathcal T}$.
	By Lemma~\ref{l4}, $\cE_{\fE}$ cannot be a hyperplane of $\mathcal T$.
	Let $\langle\varepsilon(\cT)\rangle=[T]$. Then
	$[W'']:=\langle\varepsilon(\cE_{\fE})\rangle$ has codimension at least $2$
	in $[T]$. We show that under these assumptions there would be
	an elliptic subspace $\cE$ such that
	$\cE_{\fE}\subset\cE\subseteq{\mathcal T}$ and that
	$[W'']$ has exactly codimension $2$ in $\langle\varepsilon(\cE)\rangle$.
	This would imply ${\mathfrak G}:=\fE\cup\{\cE\}\in{\mathbf C}$ and
	${\fE}\subset{\mathfrak G}$
	against the maximality of ${\fE}$ leading to a contradiction.
	
	So, let $W'\leq T$ be a subspace such that $W''$ is a hyperplane
	in $W'$. Since $\cE_{\fE}$ is elliptic, the form
	$f_{W'}:W'\times W'\to\KK$ is degenerate with
	$\mathrm{Rad}(f_{W'})\cap W''=\{\mathbf{0}\}$ by Lemma~\ref{l4}. Thus,
	$\dim(\mathrm{Rad}(f_{W'}))=1$ and
	$\mathrm{Rad}(f_{W'})=\langle\mathbf{r}\rangle$ for a
	suitable $\mathbf{r}\in W'\setminus W''$. On the other hand,
	$\mathcal T$ is elliptic; so $\mathbf{r}^{\perp_f}$ is a hyperplane of
	$T$ and since $\varepsilon(\mathcal{T})$ spans $[T]$ there is
	$\mathbf{p}\not\in\mathbf{r}^{\perp_f}$ which is singular for
	$\psi_T$. Let $W=\langle W',\mathbf{p}\rangle$ and
	$\cE:=\varepsilon^{-1}([W])$. We claim that $f_W$ is non-degenerate
	(and consequently $\cE$ is elliptic). Suppose that
	$\mathrm{Rad}(f_W)\neq\{\mathbf{0}\}$. If $\mathrm{Rad}(f_w)\subseteq W'$,
	then
	$\mathrm{Rad}(f_W)\subseteq\mathrm{Rad}(f_{W'})=\langle\mathbf{r}\rangle$;
	this would imply $\mathbf{r}\perp\mathbf{p}$, against the assumption on
	$\mathbf{p}$. So, $\mathrm{Rad}(f_W)\not\subseteq W'$ and there is
	$\mathbf{x}\in W'$ such that $\mathbf{x}+\mathbf{p}\in\mathrm{Rad}(f_W)$.
	It follows $(\mathbf{x}+\mathbf{p})\perp\mathbf{r}$; on the other hand
	$\mathbf{r}\perp\mathbf{x}$ since $\mathbf{r}\in\mathrm{Rad}(f_{W'})$
	and $\mathbf{x}\in W'$. This implies again $\mathbf{r}\perp\mathbf{p}$
	which is a contradiction. This proves the theorem.
\end{proof}


\begin{lemma}
	\label{lem pur elliptic}
	Suppose that $\cP$ has rank $n$ and  anisotropic gap $2\fd$
	and is such that all its hyperbolic lines have cardinality $2$.
	Then there exists a maximal elliptic chain
	\begin{equation}
		\label{e2}
		\fE:\cF=\cE_0\subset\cE_1\subset\dots\subset\cE_{\delta},
	\end{equation}
	of length $\fd=|\delta|$.
	Moreover, the chain $\fE$ can be enriched to a maximal chain
	$\fE^E$
	of length $2\fd$
	of nice subspaces of $\cP$ as follows
	\begin{equation}
		\label{eqn}
		\fE^E:\cF=\cE_0\subset\cP_0\subset\cE_1\subset \dots \subset \cE_i\subset \cP_i\subset \cE_{i+1}\subset \dots \subset
		\cE_{\delta}.
	\end{equation}
\end{lemma}
\begin{proof}
	By Theorem~\ref{thm hyperb lines}, $\cP$ is either a hyperbolic or an elliptic orthogonal polar space.
	If $\cP$ is of hyperbolic type, then $\fd=0$ and there is nothing to prove.
	Suppose $\cP$ to be an elliptic polar space.  
	Then $\dim(V)=2n+2\fd$.
	By Theorem~\ref{t:exists}, there are maximal well-ordered chains of
	elliptic polar spaces in $\cP$ and their maximum element is $\cP$ itself.
	Using Corollary~\ref{c1}, we can
	enrich these chains to
	well-ordered chains of nice subspaces of $\cP$
	like in~\eqref{eqn},  where the subspaces $\cE_{i}$ are all elliptic,
	and the subspaces $\varepsilon (\cP_i)$ are  described
	by a quadratic form having bilinearization with radical of dimension $1$. Moreover, $\cE_{i}$ is a hyperplane
	in $\cP_i$ and $\cP_i$ is a hyperplane in $\cE_{i+1}$.
	As a  hyperplane of a polar subspace is indeed a maximal subspace,
	it follows that
	the chain~\eqref{eqn} is a maximal well ordered chain of nice subspaces
	of $\cP$ with $2\fd$ terms.
      \end{proof}
We keep on using the notation of Setting~\ref{notazione}.
\begin{lemma}
	\label{lem:key}
	Let $\cE$ be a nice subspace of $\cP$ which is
        maximal with respect to the property that all its hyperbolic lines contain exactly two points (i.e.\ $\cE$ is a maximal elliptic subspace).
	Put $[T]:=\langle\varepsilon(\cE)\rangle$.
	Then $T$ is a complement
	of $\mathrm{Rad}(f)$ in $V$ that is
	\[ V:=T\oplus \mathrm{Rad}(f). \]
\end{lemma}
\begin{proof}
  By Corollary~\ref{e:max},
  $\cP$ contains maximal subspaces $\cE$ with the
  required properties. In light of
  Theorem~\ref{thm hyperb lines}, $\varepsilon(\cE)$ is
  either elliptic or hyperbolic. Thus $\mathrm{Rad}(f_T)=\{\mathbf{0}\}$, where $f_T\colon T\times T\rightarrow \KK$ is the bilinearization of the quadratic form $\phi_T\colon T\rightarrow \KK$ describing $\varepsilon (\cE)$ induced by $\phi$.
	In particular, $T\cap\mathrm{Rad}(f)\subseteq \mathrm{Rad}(f_T)=\{\mathbf{0}\}$.
	
	Let us consider the subspace $\langle T, \mathrm{Rad}(f)\rangle\subseteq V$.
	Suppose that $\langle T, \mathrm{Rad}(f)\rangle\neq V$.
	Then there exists a non-null vector $\mathbf{w}\in V\setminus \langle T, \mathrm{Rad}(f)\rangle$ (equivalently $\mathbf{w}+\langle T, \mathrm{Rad}(f)\rangle\in V/\langle T, \mathrm{Rad}(f)\rangle$).
	As $[V]$ is spanned by $\cP(\phi)$, we can assume
	$\phi(\mathbf{w})=0$ without loss of generality. Put $T':=\langle  T,\mathbf{w}\rangle$. Note that $T'\cap\mathrm{Rad}(f)=\{\mathbf{0}\}$,
	since otherwise $T'\subseteq\langle T,\mathrm{Rad}(f)\rangle$
        which would
	be a contradiction.
	Let now $\cE':=\varepsilon^{-1}([T'])$. Since $\cE\subseteq\cE' $ and
	$\cE$ is a maximal elliptic or hyperbolic subspace of $\cP$, the space
	$\cE'$ can be neither elliptic nor hyperbolic. Hence, the radical of the bilinear form $f_{T'}$ induced by $f$ on $T'\times T'$ is not null,
	so there exists a non-null vector
        $\mathbf{u}\in\mathrm{Rad}(f|_{T'})$, i.e.\ a
	vector $\mathbf{u}$ such that
	$f(\mathbf{u},\mathbf{t})=0$ for all $\mathbf{t}\in T'$.
        Since
	$T'\cap\mathrm{Rad}(f)=\{\mathbf{0}\}$, there is at least one element
	$\mathbf{v}\notin \mathbf{u}^{\perp_f}$ which is $\phi$-singular by~\cite[8.2.7]{T}.

	Put now $T'':=\langle T,\mathbf{u}, \mathbf{v}\rangle$ and  $\cE'':=\varepsilon^{-1}([T''])$. We have $\cE\subset \cE'\subset \cE''$.
	We show that
	$\cE''$ is elliptic or hyperbolic, thus contradicting the maximality of $\cE$.
	Indeed, suppose that $\cE''$ is neither  elliptic nor hyperbolic, so $\mathrm{Rad}(f|_{T''})\not=\{\mathbf{0}\}$.
        Then there exists a non-null vector $\mathbf{y}\in T''$ such that $f(\mathbf{y},\mathbf{x})=0$ for all $\mathbf{x}\in T''$.
	Clearly, it cannot be $\mathbf{y}\in T$, for no point in $[T]$ is orthogonal to all points of $\varepsilon(\cE)$, being $\cE$ elliptic.
        Then $\mathbf{y}=\mathbf{t}+\alpha\mathbf{u}+\beta\mathbf{v}$
	with $\mathbf{t}\in T$ and $(\alpha,\beta)\neq(0,0)$.
	If  $\mathbf{x}$ is  a generic vector in $T''$, then we can write
	$\mathbf{x}=\mathbf{t}'+\alpha' \mathbf{u}+\beta'\mathbf{v}$ for
	arbitrary $\mathbf{t}'\in T$, $\alpha',\beta'\in \KK$.
	Since $f(\mathbf{y},\mathbf{x})=0$ (and $f(\mathbf{u},\mathbf{u})=0$ because
	$\mathbf{u}\in\mathrm{Rad}(f|_{T'})$),
	we have $\forall \mathbf{t}'\in T,\alpha',\beta'\in\KK$:
	\begin{multline}
		\label{ee}
		f(\mathbf{t}+\alpha\mathbf{u}+\beta\mathbf{v},
		\mathbf{t}'+\alpha'\mathbf{u}+\beta'\mathbf{v})=\\
		f(\mathbf{t},\mathbf{t}')+\alpha' f(\mathbf{t},\mathbf{u})+
		\beta' f(\mathbf{t},\mathbf{v})+
		\alpha f(\mathbf{u},\mathbf{t}')+\alpha\beta' f(\mathbf{u},\mathbf{v})+
		\beta f(\mathbf{v},\mathbf{t}')+\beta\alpha' f(\mathbf{v},\mathbf{u})= \\
		f(\mathbf{t},\mathbf{t}')+
		\beta' f(\mathbf{t},\mathbf{v})+
		(\alpha\beta'+\alpha'\beta)f(\mathbf{u},\mathbf{v})+
		\beta f(\mathbf{v},\mathbf{t}')=0.
	\end{multline}
	Suppose now $\mathbf{t}'\in \mathbf{t}^{\perp_f}\cap\mathbf{v}^{\perp_f}$, $\beta'=0$ and $\alpha'=1$.
	Then Equation~\eqref{ee} becomes
	$\beta (f(\mathbf{u},\mathbf{v}))=0$.
	Since $f(\mathbf{u},\mathbf{v})\neq 0$, we have $\beta=0$;
	so $\mathbf{y}=\mathbf{t}+\alpha\mathbf{u}$, hence
	$\mathbf{t}=\mathbf{y}-\alpha\mathbf{u}$.  But
	$\mathbf{y}^{\perp_f}\supseteq T''\supset T'\supset T$ and
	$\mathbf{u}^{\perp_f}\supseteq T'\supset T$. Hence $\mathbf{t}^{\perp_f}
		\supseteq T$, so $\mathbf{t}\in\mathrm{Rad}(f_T)$.
	Since $\cE$ is elliptic, $f_T$ is non-degenerate; so
	$\mathbf{t}=\mathbf{0}$.
        So,
	Equation~(\ref{ee}) becomes $\alpha \beta' f(\mathbf{u},\mathbf{v})=0$. Choose $\beta'=1$, so $\alpha=0$.
	This is a contradiction. So, the space $\cE''$ is elliptic.
	This is again a contradiction because $\cE\subset \cE''$ and $\cE$ is a maximal elliptic subspace of $\cP$.
	Hence we have that $\langle T, \mathrm{Rad}(f)\rangle$ cannot be properly contained in $V$.  This implies the lemma.
\end{proof}
We now conclude the proof of Theorem~\ref{m:1} by
proving that all the maximal elliptic chains of a classical orthogonal polar space have the same length.
\medskip

By Corollary~\ref{e:max}, $\cP$ contains a maximal nice subspace with respect
to the property that its hyperbolic lines consist of just $2$ points;
thus it is either the subspace generated by a frame or a maximal elliptic
subspace of $\cP$; denote it by $\cE$
and let $[T]=\langle\varepsilon(\cE)\rangle$.
Then
$\varepsilon|_{\cE}:\cE\to [T]$ is the universal
embedding of $\cE$. By Lemma~\ref{lem:key}, $T$ is a direct complement of the radical of $f$. The subspace $\mathrm{Rad}(f)$, and hence $T\cong V/\mathrm{Rad}(f)$, is uniquely
determined by $\cP$ and $\varepsilon$ and does
not depend on the maximal elliptic subspace
of $\cP$ we have chosen.
So, fix  any  maximal elliptic subspace $\cE$ of $\cP$.
By Lemma~\ref{lem pur elliptic} with $\cE$ in the role of $\cP$,
there exists a maximal elliptic chain of length $\delta$ in $\cE$ which refines to a maximal chain of nice
subspaces as in~\eqref{eqn} of length $2\delta$. In particular, the length of the
chain~\eqref{eqn}  is the same as the anisotropic gap of $\cE$ which is the dimension of
$\dim(V/R)-2n$, or equivalently, the codimension of $\mathrm{Rad}(f)$
in $V_0$.\\
\medskip
\noindent  {\bf This ends the proof of Theorem~\ref{m:1}}. \hfill$\Box$

\begin{remark}
	It is a consequence of Theorem~\ref{lem:key} that the codimension
	in $\PG(V)$ of 	the image of the embedding of a maximal elliptic
        subspace $\cE$ of $\cP$ does
	not depend on the choice of $\cE$.
	In particular, the image of the universal embedding of any maximal
	elliptic subspace of $\cP$ is hosted in $\PG(V/\mathrm{Rad}(f))$,
	where $\PG(V)$ is the codomain of the universal embedding of
	$\cP$.
	When the dimension of $V/\mathrm{Rad(f)}$ is finite, it is possible
	to use the standard theory of quadratic forms (Witt's theorem) to
	prove that all maximal elliptic subspaces of $\cP$ must be
	isomorphic to each other (even if there might be elliptic subspaces
	of $\cP$ which are non-isomorphic).
	We leave the case of infinite dimension (in particular,
	when the dimension of this space is larger that $\aleph_0$ to future
	investigation).
\end{remark}

\section{Parabolic gap}\label{parabolic gap}

As in Section~\ref{sec 3}, assume that $\cP$ is an orthogonal polar space defined over a field of characteristic $2$. If $\cP$ is a grid, then
let $\varepsilon\colon\cP\rightarrow\PG(V)$ be any of its relatively universal
embeddings. Otherwise, let
$\varepsilon\colon \cP \rightarrow \PG(V)$ be the universal embedding of
$\cP$.
In any case we can write $\varepsilon(\cP)=\cP(\phi)$, where $\phi\colon V\rightarrow \KK$ is a quadratic form having bilinearization $f\colon V\times V\rightarrow \KK$ (see Section~\ref{sesquilinear and pseudoquadratic form}).

\subsection{Proof of Theorem~\ref{m:2}}
Suppose that $\cP$ has elliptic gap $\fd$.
Let
\[\fE: \cF=\cE_0\subset\cE_1\subset\dots\subset\cE_{\delta}\]
be a maximal elliptic chain of $\cP$ (of length $\fd=|\delta|$) and put
$\cE':=\bigcup_{\cE\in\fE}\cE$.
Then  $\cE'$ is a maximal subspace of $\cP$ with the property that all its hyperbolic lines have $2$ points and, by Lemma~\ref{lem:key},
$V_0=\langle\varepsilon(\cE')\rangle\oplus\mathrm{Rad}(f)$.
So,
$\fd=\dim(V_0/\mathrm{Rad}(f))$
with $\fr:=\dim(V_0)$ being the anisotropic gap of $\cP$.

Using
Lemma~\ref{parabolic defect}
we can extend the chain $\fE$ to a maximal well ordered chain
of nice subspaces of $\cP$, say $\fN$. By Theorem~\ref{plain}, the length of $\fN$ (being an invariant of $\cP$) is  precisely the anisotropic gap $\fr$ of $\cP$ and $\fr=\dim (V_0)$. Hence $\dim({\mathrm{Rad}}(f))=\fr-\fd$
in the sense of Remark~\ref{r:111}.
\par\noindent\ \hfill$\Box$

\subsection{Parabolic polar spaces}\label{section finale}
In light of Corollary~\ref{d112} and our definitions, $\cP$ is parabolic
when it is $(0,\fp)$-orthogonal with $\fp>0$; that is
$\cP$
properly contains a hyperbolic polar space as a subspace
and its elliptic gap is $0$.
As usual, let $\varepsilon:\cP\to\PG(V)$ be the universal embedding
of $\cP$. Then $\cP$ is parabolic if we have the
following decomposition of $V:$
\begin{equation}\label{pp}
	V=(V_1\oplus V_2\oplus\dots\oplus V_n)\oplus\mathrm{Rad}(f).
\end{equation}
Note that the quotient embedding $\varepsilon/\mathrm{Rad}(f):\cP\to\PG(V/\mathrm{Rad}(f))$
has dimension $2n$ and it is described by a generalized quadratic form (see~\cite{P}) or by an alternating form.
By~\cite[Theorem 1.11]{P22}, the condition of admitting an embedding of dimension $2n$ for $\cP$ can be formulated as follows:
\begin{itemize}
	\item[(A)] for any two non-collinear points $a,b\in\cP$, if  $M$ and $N$
		are respectively  maximal singular subspaces of
		$\cP$ and $\{a,b\}^\perp$ and $N\subseteq M$, then $M\cap\{a,b\}^{\perp\perp}\neq\emptyset$.
\end{itemize}

\begin{remark}
	\label{rr}
	Observe that a hyperbolic polar space is an embeddable polar
	space whose hyperbolic lines consist of just $2$ points
	which satisfies Condition~(A).
	Indeed, Condition~(A) by itself
	assures the existence of an embedding of minimum
	dimension (i.e.\ $2n$).
	If all the hyperbolic lines of $\cP$  have
	just $2$ points, then such an embedding cannot be obtained as
	a quotient of an embedding of higher dimension.
	Hence it is unique
	(except for the case of grids). Moreover, the
	condition on the cardinality
	of the hyperbolic lines guarantees that this embedding realizes
	$\cP$ as the point set of
	a hyperbolic quadric; this also in the case of the grids,
        even if the embedding is not unique. So, a grid is always
	a hyperbolic polar space, even if it does not admit
	a universal embedding.
\end{remark}
Thanks to Remark~\ref{rr},
we can rephrase the definition of parabolic spaces already given in  Section 1 as follows.

\begin{definition}
	A \emph{parabolic polar space} $\cP$ is an embeddable polar
	space properly containing a hyperbolic polar subspace and satisfying
	Condition~(A).
\end{definition}

\begin{theorem}
	Let $\cP$ be an orthogonal polar space.
	Then its parabolic gap  is at most
	$[\KK:\KK^2]$.
\end{theorem}
\begin{proof}
	Suppose $\cP$ to have rank $n$.
	If $\cP$ is hyperbolic or elliptic, then there is nothing to prove.
	Otherwise, put
	$R:=\mathrm{Rad}(f)$. We claim that
	$\dim(R)\leq [\KK:\KK^2]$.
	Indeed, let $(\mathbf{r}_i)_{i\in I}$ be a basis of $R$.
	As $R$ is anisotropic, for any $x_i\in \KK\setminus\{0\}$
	we have
	\[ \sum_{i\in I} \phi(\mathbf{r_i})x_i^2\neq 0. \]
	In particular, the values $\xi_i:=q(\mathbf{r_i})\in\KK$ must be linearly
	independent over $\KK^2$, where $\KK$ is regarded as a vector space over its subfield of squares $\KK^2$. It follows that we have $\dim(R)\leq[\KK:\KK^2]$. So, by Theorem~\ref{m:2}, the parabolic gap of $\cP$ is at most $[\KK:\KK^2]$.
\end{proof}

\begin{remark}
	\label{c:cpp}
	Let $\cP$ be an orthogonal polar space
	and suppose   $[\KK:\KK^2]<\infty$ with $\mathrm{char}(\KK)=2$.
	By our results, the image of the universal embedding of $\cP$ decomposes
	as in~\eqref{decomposition} and the parabolic and elliptic gaps of $\cP$
	are respectively $\fp=\dim(\mathrm{Rad}(f))$ and $\fe=\dim(V_0')$.
	By standard results in the theory of quadratic forms,
        see~\cite[Lemma 36.8]{EKM}
	we have $\frac{1}{2}\fe+\fp\leq[\KK:\KK^2]$
        (when $\fe$ is infinite we have $\frac{1}{2}\fe=\fe$).
	In particular, if
	the parabolic gap $\fp$ of $\cP$ is
	exactly $[\KK:\KK^2]$, then $\fe=0$.
	In this case, the
	(minimum)
	embedding of $\cP$ (obtained by quotienting its universal
	embedding over $\mathrm{Rad}(f)$) realizes $\cP$ as a symplectic space.
	The converse is also true:
	if $\cP$ is a polar space that admits an embedding over
	a field $\KK$ with $\mathrm{char}(\KK)=2$ described by an alternating
	bilinear form, then $\cP$ is parabolic of gap $[\KK:\KK^2]$.
\end{remark}

\section*{Thanks}
The authors wish to express their sincere
gratitude to Prof.\ Antonio Pasini for his
precious advice and remarks on this work.

Both authors are affiliated with GNSAGA of INdAM (Italy) whose support
they acknowledge.

\vskip.2cm\noindent
\begin{minipage}[t]{\textwidth}
	Authors' addresses:
	\vskip.2cm\noindent\nobreak
	\centerline{
		\begin{minipage}[t]{7cm}
			Ilaria Cardinali,\\
			Department of Information Engineering and Mathematics\\University of Siena\\
			Via Roma 56, I-53100, \\ Siena, Italy\\
			ilaria.cardinali@unisi.it,\\
		\end{minipage}\hfill
		\begin{minipage}[t]{5cm}
			Luca Giuzzi\\
			D.I.C.A.T.A.M. \\
			University of Brescia\\
			Via Branze 43, I-25123, \\ Brescia, Italy \\
			luca.giuzzi@unibs.it
		\end{minipage}}
\end{minipage}

\end{document}